\documentclass[12pt, reqno]{amsart}
\usepackage{amsmath, amstext, amsbsy, amssymb, amscd}
\usepackage{amsxtra}
\usepackage{amscd}
\usepackage{amsthm}
\usepackage{amsfonts}
\usepackage{eucal}
\usepackage{color}
\usepackage[all]{xy}
\usepackage[CJKbookmarks=true]{hyperref}

\usepackage{amsmath}
\usepackage{amsfonts}
\usepackage{amssymb}
\newtheorem{theorem}{Theorem}[section]
\newtheorem{lemma}[theorem]{Lemma}
\newtheorem{prop}[theorem]{Proposition}
\newtheorem{corollary}[theorem]{Corollary}
\theoremstyle{definition}

\newtheorem{remark}[theorem]{Remark}

\numberwithin{equation}{section}

\def\Ann{\mathrm{Ann}}

\def\ad{\mathrm{ad}}

\def\End{\mathrm{End}}
\def\Ch{\mathrm{Ch}}

\def\Lie{\mathrm{Lie}}

\def\fin{\mathrm{fin}}

\def\gr{\mathrm{gr}}

\def\Prim{\mathrm{Prim}}

\def\ggg{\mathfrak{g}}

\makeatletter
\newcommand{\rmnum}[1]{\romannumeral #1}
\newcommand{\Rmnum}[1]{\expandafter\@slowromancap\romannumeral #1@}
\makeatother

{\vskip-\lastskip\medskip
  \noindent
  {\em #1.}\enspace
  }%
{\qed\par\medskip
  }

\begin{document}

\title[On finite dimensional  representations  of finite W-superalgebras ]{On finite dimensional representations of finite W-superalgebras}

\author{Husileng Xiao}

\address{ School of mathematical Science, Harbin Engineering University, Harbin, 150001, China }\email{hslxiao@hrbeu.edu.cn}
\begin{abstract}
We first formulate and prove a version of Premet's conjecture for  finite W-superalgebras associated to basic Lie superalgebras. As in the case of W-algebras,  Premet's conjecture is very close to giving a classification of finite dimensional simple modules of finite W-superalgebras.
In the case of basic type \Rmnum{1} Lie superalgebras, 
we classify the finite dimensional simple supermodules  with integral central character and give an algorithm 
 to compute their characters based on the $\ggg_{\bar{0}}$-rough structure of $\ggg$-modules. 
\end{abstract}
\maketitle
\section{Introduction}

\let\thefootnote\relax\footnotetext{ Keywords : W-superalgebras, Premet's conjecture, finite dimensional representations}
\let\thefootnote\relax\footnotetext{ MSC: 17B10 17B63 17B69}

Finite W-superalgebras are the Zhu algebras of affine W-superalgebras in the sense of \cite{DSK}.
The latter includes the well-known $N=1,2,3,4$ superconformal algebras and plays a very important role in supersymmetric quantum field theory.
The affine W-superalgebras were constructed in \cite{KRW} by quantum Hamiltonian reduction in the general setting. 
However, the finite W-superalgebras appear in mathematics more indirectly.
Generalizing the groundbreaking work \cite{Pr1},  Wang and Zhao \cite{Zh,WZ} first studied finite W-superalgebras from the viewpoint of modular Lie superalgebras. Here the term modular means that the ground algebraically closed field has a positive characteristic.

Let $\ggg=\ggg_{\bar{0}}+\ggg_{\bar{1}}$ be a basic Lie superalgebra, $\mathcal{W}_0$ (resp. $\mathcal{W}$) be the finite W-(resp. super-)algebra constructed from a fixed  nilpotent element  $ e \in \ggg_{\bar{0}}$. Based on a relation between the finite W-algebra $\mathcal{W}_0$ and W-superalgebra $\mathcal{W}$  found recently by the author and Shu,  we study the finite dimensional irreducible representations of finite W-superalgebras in this paper. Let $\mathrm{Irr}^{\mathrm{fin}}(\mathcal{W})$ stand for the set of isomorphism classes of irreducible modules.

   Brown, Brundan and Goodwin \cite{BBG,BG} gave  a Yangian presentation of W-superalgebras corresponding to  principal nilpotent elements in the general linear Lie superalgebras.
Relying on this explicit presentation, they gave a description of $\mathrm{Irr}^{\mathrm{fin}}(\mathcal{W})$ and further detailed information on their highest weight structures.  Poletaeva and Serganova  \cite{PS1} proved an Amitsur-Levitzki identity for the  W-superalgebras associated to  principal nilpotent elements in the queer Lie superalgebras. They obtained that any irreducible
representation is finite dimensional. These results indicate that the representation theory of  finite W-superalgebras is quite different from that of finite W-algebras.
By giving an explicit description of the structure of  W-superalgebras associated to  minimal nilpotent elements,  Zeng and Shu \cite{ZS} constructed their irreducible representations with  dimension 1 or 2. Recently, Chen \cite{Ch} investigated the Whittaker category $\mathcal{N}$ for  basic Lie superalgebras. Through Skryabin's equivalence, the category $\mathcal{N}$  is equivalent to $\mathcal{W}$-Mod  when $\mathcal{W}$ is associated to a principal nilpotent $e$.

 However, unlike  the case of finite W-algebras,  some fundamental problems in the representation theory of finite W-superalgebras are still open in general. 
 In \cite{SX}, Shu and the author generalized  Losev's  Poisson geometric approach to the super case and made a  step to give a classification of finite dimensional irreducible representations of  finite W-superalgebras in a more general setting. In this article we make a progress to this problem  by proving  Premet's  conjecture for the W-superalgebras of basic Lie superalgebras. In particular, we classify the finite dimensional simple $\mathcal{W}$-supermodules with integral central character and obtain an algorithm  to compute their characters in the basic type \Rmnum{1} case.
 
 We hope that the readers could be convinced that the difference between finite W-algebras and  W-superalgebras probably not exceeds that between  Lie algebras and Lie superalgebras.

\subsection{Premet's conjecture for finite W-superalgebras}\label{sec1.1}
Let $\ggg=\ggg_{\bar{0}}\oplus \ggg_{\bar{1}}$ be a basic Lie superalgebra over an algebraically closed field $\mathbb{K}$ with $\mathrm{Char}(\mathbb{K})=0$, $\mathcal{U}$ and $\mathcal{U}_{0}$ be the enveloping algebra of $\mathfrak{g}$ and $\ggg_{\bar{0}}$ respectively. Denote by $(\bullet, \bullet)$ the Killing form on $\ggg$. Fix a nilpotent $e \in \ggg_{\bar{0}}$ and  let $\chi \in \ggg_{\bar{0}}^{*}$ be the corresponding element to $e$ via the Killing form. Pick an $\mathfrak{sl}_2$-triple $\{ f,h,e \} \subset \ggg_{\bar{0}}$ and let $\ggg=\bigoplus_{i}\ggg(i)$ (resp. $\ggg_{\bar{0}}=\bigoplus_{i}\ggg(i)\cap\ggg_{\bar{0}}$) be the $\mathbb{Z}$-grading given by
the adjoint action of $h$. Denote by $\mathcal{W}$ and $\mathcal{W}_{0}$ the W-algebras associated to the pairs $(\ggg,e)$ and $(\ggg_{\bar{0}},e)$ respectively.  Let $\tilde{\mathcal{W}}$ be the extended W-superalgebra  $\mathcal{A}_{\ddag}$ defined in  \cite[\S3]{SX}(Note that it is denoted by $\mathcal{A}_{\dag}$ in \cite[\S6]{Lo15}).
The following relation among the three kinds of W-algebras was found in \cite{SX}:  (1) we have an embedding
$\mathcal{W}_{0} \hookrightarrow \tilde{\mathcal{W}}$ and the latter is generated over the former by $\dim(\ggg_{\bar{1}})$ odd elements;
(2) we have an isomorphism $\tilde{\mathcal{W}}\simeq \mathrm{Cl}(V_{\bar{1}})\otimes \mathcal{W}$ of associative algebras, where $\mathrm{Cl}(V_{\bar{1}})$ is the Clifford algebra over a vector space $V_{\bar{1}}$ with a non-degenerate symmetric bilinear form; see Theorem \ref{them1} for the details.
 Essentially, as mentioned  in \cite{SX},  this  makes $\mathcal{W}_{0}$  to play a role in the representation theory of  $\mathcal{W}$ as $\mathcal{U}_{0}$ does in that of $\mathcal{U}$. The 
 representation theories of  $\mathcal{W}$ and $\tilde{\mathcal{W}}$ are equivalent; see Proposition \ref{prop2.5}. However, as we will see in the present work,
 a significant advantage to consider  $\tilde{\mathcal{W}}$ instead of $\mathcal{W}$ is that it is easy to relate $\tilde{\mathcal{W}}$ with $\mathcal{W}_{0}$.
 This enables us to use results on $\mathcal{W}_{0}$.

Given an associative algebra $\mathcal{A}$, we denote by $\mathfrak{id}(\mathcal{A})$  the set of two-sided ideals of $\mathcal{A}$ and by $\mathrm{Prim}^{\mathrm{fin}}(\mathcal{A})$  the set of primitive ideals of $\mathcal{A}$ with finite codimension. It is well known that $\mathrm{Prim}^{\mathrm{fin}}(\mathcal{A})$ is bijective with the set $\mathrm{Irr}^{\mathrm{fin}}(\mathcal{A})$ of isomorphism classes of finite dimensional irreducible $\mathcal{A}$-modules.
 Losev \cite{Lo10}  constructed an ascending map $\bullet^{\dag} : \mathfrak{id}(\mathcal{W}_{0}) \longrightarrow \mathfrak{id}(\mathcal{U}_{0})$ and a descending map $\bullet_{\dag} : \mathfrak{id}(\mathcal{U}_{0}) \longrightarrow \mathfrak{id}(\mathcal{W}_{0})$. These two maps are crucial to his study on the representations of $\mathcal{W}_{0}$.
The ascending map $\bullet^{\dag}$ sends $\mathrm{Prim}^{\mathrm{fin}}(\mathcal{W}_{0})$ to the set $\mathrm{Prim}_{\mathbb{O}}(\mathcal{U}_{0})$ of primitive ideals of $\mathcal{U}_{0}$ supported on the Zariski closure of the adjoint orbit $\mathbb{O}=G_{\bar{0}} \cdot \chi $.
 Denote by  $Q=Z_{G_{\bar{0}}}\{ e,h,f \}$ the stabilizer of the triple $\{ e,h,f \}$ in $G_{\bar{0}}$ under the adjoint action.
 Let $C_{e}=Q/Q^\circ$, where $Q^\circ$ is the identity component of $Q$.  Premet's conjecture which was proved in \cite{Lo11}, states that for any
 $\mathcal{J} \in \mathrm{Prim}_{\mathbb{O}}(\mathcal{U}_{0}) $ the set $\{ \mathcal{I} \mid  \mathcal{I} \in \mathrm{Prim}^{\mathrm{fin}} (\mathcal{W}) ,  \mathcal{I}^{\dag}=\mathcal{J} \}$ is a single $C_{e}$-orbit.
 This indicates to us an  almost complete classification of $\mathrm{Irr}^{\mathrm{fin}}(\mathcal{W}_{0})$.

In this paper, we generalize the above fact to the super case. Recall that the super analog of the maps $\bullet^{\dag}$ and $\bullet_{\dag}$ was established in \cite{SX}. By abuse of notation, we also denote them by $\bullet^{\dag}$ and $\bullet_{\dag}$ from now on.
Denote by $\mathrm{Prim}_\mathbb{O}(\mathcal{U})$ the set of primitive ideals of $\mathcal{U}$ supported on the Zariski closure of $\mathbb{O}$; see \S 2 for the definition of  `supported' in the super context.
In \S 2 we will construct an action of $Q$  on $\tilde{\mathcal{W}}$ with a property that $Q^{\circ}$ leaves any two-sided ideal of $\tilde{\mathcal{W}}$ stable; see Proposition \ref{prop2.1}. This yields an action of $C_e$  on $\mathfrak{id}(\tilde{\mathcal{W}})$.

We also consider the $\mathbb{Z}_2$-graded version of the above setting. 
For a superalgebra $\mathcal{A}=\mathcal{A}_{\bar{0}}+\mathcal{A}_{\bar{1}}$,  the $\mathbb{Z}_2$-graded $\mathcal{A}$-modules will be called $\mathcal{A}$-supermodules.   
An ideal $\mathcal{I}$ of $\mathcal{A}$ is said to be graded primitive if it is the annihilator of a simple object in the category of $\mathcal{A}$-supermodules.
Denote by $\gr.\mathrm{Prim}(\mathcal{A})$ the set of graded primitive ideals of $\mathcal{A}$. For a notation $\bullet$ used in the ungraded case, we always use $\gr.\bullet$ in the $\mathbb{Z}_2$-graded case by the same way as above.  Since the  action of $Q$ on $\tilde{\mathcal{W}}$ is $\mathbb{Z}_2$-homogeneous,  we also have an action of  $C_e$ on  $\gr.\mathfrak{id}(\tilde{\mathcal{W}})$. 
Our first main result reads:
\begin{theorem} \label{mainthm}
For any $\mathcal{J} \in \mathrm{Prim}_\mathbb{O}(\mathcal{U})$, the set $\{ \mathrm{Cl}(V_{\bar{1}}) \otimes \mathcal{I} \mid  \mathcal{I} \in \mathrm{Prim}^{\mathrm{fin}} (\mathcal{W}), \mathcal{I}^{\dag}=\mathcal{J} \}$  consisting of 
the primitive ideals of $\mathcal{W}$ lying over  $\mathcal{J}$,  
 is a single $C_e$-orbit. For any $\mathcal{J} \in \gr.\mathrm{Prim}_\mathbb{O}(\mathcal{U})$, the set consisting of the graded primitive ideals of $\mathcal{W}$ 
 lying over  $\mathcal{J}$ is also a single $C_e$-orbit.
\end{theorem}

We  also have maps $ \bullet^{\tilde{\dag}}: \mathfrak{id}(\tilde{\mathcal{W}}) \rightarrow  \mathfrak{id}(\mathcal{U})$  and 
$ \bullet_{\tilde{\dag}}:  \mathfrak{id}(\mathcal{U})  \rightarrow \mathfrak{id}(\tilde{\mathcal{W}}) $ which can be defined similarly to $ \bullet^{\dag}$ 
and  $ \bullet_{\dag}$; see Lemma \ref{lem2.4}.  Theorem \ref{mainthm} is equivalent to saying that the set $\{  \mathcal{\tilde{I}} \mid \mathcal{\tilde{I}}\in \mathrm{Prim}^{\mathrm{fin}} (\tilde{\mathcal{W}}), \mathcal{\tilde{I}}^{\tilde{\dag}}=\mathcal{J} \}$ of 
primitive ideals lying over  $\mathcal{J}$,  is a single $C_e$-orbit.

The strategy of our proof  is  applying  \cite[Theorem 4.1.1]{Lo11} to the Harish-Chandra bimodule  $\mathcal{U}$ over $\mathcal{U}_{0}$
and the relation among $\mathcal{W}$, $\mathcal{W}_{0}$ and $\tilde{\mathcal{W}}$ introduced previously. This is
highly inspired by  \cite[\S 6]{Lo15}.

We can recover $\mathcal{I}$  from  $\mathrm{Cl}(V_{\bar{1}}) \otimes \mathcal{I}$ by Corollary \ref{coro2.4}.
It is known that  the map $\bullet^\dag$ sends $\mathrm{Prim}^{\mathrm{fin}}(\mathcal{W})$ to $\mathrm{Prim}_{\mathbb{O}}(\mathcal{U})$; see \cite[Theorem 4.8]{SX}. So Theorem \ref{mainthm}  almost completely reduces the problem of classifying  $\mathrm{Prim}^{\mathrm{fin}} (\mathcal{W})=\mathrm{Irr}^{\mathrm{fin}}(\mathcal{W})$ to that of $\mathrm{Prim}(\mathcal{U})$.   If we know   $\mathrm{Prim}(\mathcal{U})$ and $C_e$ is trivial, Theorem \ref{mainthm}  gives  a description of   $\mathrm{Irr}^{\mathrm{fin}}(\mathcal{W})$; see \S \ref{2.4}.   For the recent progress on the primitive ideals of Lie superalgebras, see for examples \cite{CoM} and \cite{Mu97}. 

We say that  $ M \in \mathrm{Irr}^{\mathrm{fin}}(\tilde{\mathcal{W}})$ (or $M' \in \mathrm{Irr}^{\mathrm{fin}}(\mathcal{W})$) lies over a primitive ideal 
$\mathcal{J}$ of $\mathcal{U}$ if so do their annihilators. It is well known that  for the basic classical Lie superalgebras $\ggg$, any primitive ideal of  $\mathcal{U}$ is the annihilator $\widehat{J}(\lambda)$ of a highest weight simple module $\widehat{L}(\lambda)$ for some $\lambda \in \mathfrak{h}^* $. We say that a finite dimensional simple $\mathcal{\tilde{W}}$-module has \textit{center character} $\lambda $  if it lies over $\widehat{J}(\lambda)$. 
Let  $\mathrm{Irr}_\lambda(\tilde{\mathcal{W}})$ stand for the set of isomorphism classes of $\tilde{\mathcal{W}}$-supermodules 
with center character $\lambda$. Define  $\mathrm{Irr}_\lambda(\mathcal{W}_0)$,  $\mathrm{Irr}_\lambda^{\mathrm{fin}}(\mathcal{W}_0)$  and  $\mathrm{Irr}_\lambda^{\mathrm{fin}}(\mathcal{W})$ similarly. 
Theorem \ref{mainthm} gives us an action of  $C_e$ on  $\mathrm{Irr}_\lambda^{\mathrm{fin}}(\tilde{\mathcal{W}})$ and $\mathrm{Irr}_\lambda^{\mathrm{fin}}(\mathcal{W})$; see \S \ref{sec2.5}.

\subsection{ Finite dimensional  representations of basic type \Rmnum{1} W-superalgebras}

In the remaining  part of this section, let $\ggg=\ggg_{\bar{0}}+\ggg_{\bar{1}}$ be a basic type \Rmnum{1} simple Lie superalgebra.
Namely,  $\ggg$ is one of the following list, Type (A): $\mathfrak{gl}(m|n), \mathfrak{sl}(m|n), \\ \mathfrak{sl}(n|n)/\mathbb{C}I_{n|n}$; 
Type(C): $\mathfrak{osp}(2|2n)$.  

A classification of simple $\ggg$-supermodules was obtained in \cite{CM}.
It was proved  that there is an one-to-one correspondence between 
the set of isomorphism classes of  simple $\ggg$-supermodules  and 
that of the simple $\ggg_{\bar{0}}$-modules. Given a simple $\ggg_{\bar{0}}$-module $V$,  we denote by  $\widehat{V}$ the
simple $\ggg$-supermodules under this correspondence, which is the unique simple quotient of the Kac module $K(V)$.
This result is fundamental to the present work.
Using Skryabin's equivalence, we  decent this result to the context of W-algebras.  
More precisely, we prove that the sets 	
$\mathrm{Irr}(\mathcal{W}_0)$, $\gr.\mathrm{Irr}(\tilde{\mathcal{W}})$ and $\gr.\mathrm{Irr}(\mathcal{W})$ are bijective with each other. By abuse of notation, for a simple $\mathcal{W}_0$-module $N$, we also denote by $\widehat{N}$ the unique simple $\tilde{\mathcal{W}}$-supermodule under this correspondence. 
However, this classification of $\mathrm{Irr}(\tilde{\mathcal{W}})$ is not well organized.
For example, it is difficult  to see the behavior of  the action of  $C_e$ under the correspondence. To fix up this problem, we give another better  classification of $\gr.\mathrm{Irr}(\tilde{\mathcal{W}})$.
 To that end, we present  a triangular decomposition $ \mathcal{\tilde{W}}=\mathcal{\tilde{W}}_{+}^{\#}\otimes_{\mathbb{K}}\mathcal{W}_0\otimes_{\mathbb{K}}  \mathcal{\tilde{W}}_{-}^{\#}$ for $\tilde{\mathcal{W}}$. This can be compared with the decomposition $\ggg=\ggg_{-1} +\ggg_{0} +\ggg_{1} $ of the type \Rmnum{1} simple Lie superalgebras. A crucial point  is that  $\mathcal{W}_0$ is the ordinary finite W-algebra from $(\ggg_{\bar{0}},e)$. Using this decomposition,
for any finite dimensional simple $\mathcal{W}_0$-module $N$, we define `Verma' module $\Delta^{K}_{\mathcal{\tilde{W}}}(N)$ over $\mathcal{\tilde{W}}$ and prove that it has a unique simple
$\mathbb{Z}_2$-graded quotient $L^{K}_{\mathcal{\tilde{W}}}(N)$. We point out that  it is easy to  obtain  a triangular decomposition
$\mathcal{W}= \mathcal{W}_{-}^{\#}\otimes_{\mathbb{K}}\mathcal{W}_0'\otimes_{\mathbb{K}}   \mathcal{W}_{-}^{\#}$ for the usual finite W-superalgebra $\mathcal{W}$ by a similar method in here. A triangular decomposition  has already been obtained  for $\mathcal{W}$ arising from the general Lie superalgebras by using super Yangian presentation; see  \cite{BBG} for the principal nilpotent element $e$  and \cite{Pe} for the general case. Compared with the one for $\mathcal{\tilde{W}}$,  a disadvantage of the latter is that it is highly non-trivial to relate $\mathcal{W}_0'$ and $\mathcal{W}_0$ for general $e$,
although the two algebras coincide  when  $e$ is principal nilpotent element.

Our main tool used to compute the character of  simple $\mathcal{\tilde{W}}$-modules with integral center character
is the generalized Soergel functor $\mathbb{V}$ for $\mathcal{W}_0$ constructed in \cite{Lo15}. Let $P \subset G_{\bar{0}}$( resp. $\mathfrak{p}=\Lie(P)$) be the suitable parabolic subgroup (resp. subalgebra) constructed from an $\mathfrak{sl}_2$-triple in \cite{Lo15}. Denote by $\mathcal{O}^P$ the corresponding parabolic category $\mathcal{O}$
 and $\Lambda_\mathfrak{p}$ the set consisting of the integral  $\lambda \in \mathfrak{h}^*$ such that the 
highest weight simple module $L(\lambda)$ lies in $\mathcal{O}^P$. Let $\mathbb{V}: \mathcal{O}^P \rightarrow \mathcal{O}_\theta(\ggg_{\bar{0}},e)$ be the
generalized Soergel functor for $\mathcal{W}_0$ defined in \cite{Lo15}. The notation will be 
recalled in \S \ref{sec cher for}.
Let $\lambda \in \Lambda_\mathfrak{p} $ with $\mathbb{V}(L(\lambda))\neq 0$. Describing $\mathbb{V}(L(\lambda))$,  Losev established a character formula for the modules in $\mathrm{Irr}_\lambda^{\mathrm{fin}}(\mathcal{W}_0)$ with integral $\lambda$.  His character formula is based on the parabolic Kazhdan-Lusztig theory for $\mathcal{O}^P$. We will give a description of  $\mathbb{V}(\widehat{L}(\lambda))$ for the simple $\ggg$-supermodules $\widehat{L}(\lambda) \in \mathcal{O}^P $. Relying on the $\ggg_{\bar{0}}$-rough structure of simple $\ggg$-supermodules,  we  compute the characters of modules in $\mathrm{Irr}_\lambda^{\mathrm{fin}}(\mathcal{W})$ for integral  $\lambda$.  Note that,   just like the even case,  the set $\mathrm{Irr}_\lambda^{\mathrm{fin}}(\mathcal{W})$ is non-empty only if $\lambda \in \Lambda_{\mathfrak{p}}$.

In summary, for $\ggg$ being a basic type \Rmnum{1} Lie superalgebra, our main results are as follows.
\begin{itemize}
	
   \item[(1)] We obtain a triangular decomposition for  $\mathcal{\tilde{W}}$ and some standard properties of  Verma modules defined by it. 
	We prove that the map $ \mathrm{Irr}^{\mathrm{fin}}(\mathcal{W}_0) \rightarrow \gr.\mathrm{Irr}^{\mathrm{fin}}( \tilde{\mathcal{W}}): N \mapsto L^{K}_{\mathcal{\tilde{W}}}(N)$ is bijective and $C_e$-equivariant;
	see Proposition \ref{trian decom}. As an application, we also prove that 
	$\gr.\mathrm{Prim}(\mathcal{W})$ is bijective with $\mathrm{Prim}(\mathcal{W}_0)$; see Corollary \ref{Coro4.2}.
	\item[(2)] For  $ \lambda \in \Lambda_{\mathfrak{p}}$,   let $\mathbb{V}(L(\lambda))=\bigoplus_{i \in I_{\lambda}} N_i $ be the description of $\mathbb{V}(L(\lambda))$ obtained in \cite{Lo15}. Here $I_{\lambda}$ is a finite set and $N_i \in \mathrm{Irr}_\lambda^{\mathrm{fin}}(\mathcal{W}_0)$ for $i \in I_{\lambda} $. Then we have $$\mathbb{V}(\widehat{L}(\lambda))=\bigoplus_{i \in I_{\lambda}} L^{K}_{\mathcal{\tilde{W}}}(N_i);$$  see  Theorem \ref{Thm sger func}.
	\item[(3)] For integral $\lambda$, we will present an algorithm
to compute the characters of modules in $\mathrm{Irr}_\lambda^{\mathrm{fin}}(\mathcal{W})$;  see \S \ref*{sec4.5}. 
	\end{itemize}

Finally, we point out that the powerful tools about W-algebras developed by Losev 
   will be used in the whole paper. However, they are very technical and rely heavily on the geometry of nilpotent orbits.  In the super case, we will overcome these difficulties (see Proposition \ref{prop2.1} and \ref{Triangular Decom}).

\section{ A super version of Premet's conjecture}

We first recall the definition of finite W-(super)algebras in the sense of Premet.
We continue with the notation from  Section \S \ref{sec1.1}.
Let $l=l_{\bar{0}}+l_{\bar{1}}$ be a Lagrangian subspace 
of $\ggg(-1)$ with respect to the super symplectic form $\chi([,])$. Thus $l_{\bar{0}}$ is automatically a Lagrangian subspace of 
 $\ggg_{\bar{0}}(-1)$.  Set 

$$\mathfrak{m}=\bigoplus_{i\leq -2}\ggg(i) \oplus l, \mathfrak{m}_{\bar{0}}=\bigoplus_{i\leq -2}\ggg_{\bar{0}}(i) \oplus l_{\bar{0}}$$ and 
$$ \mathfrak{m}_\chi=\{ x-\chi(x)\mid  x \in \mathfrak{m} \},\mathfrak{m}_{\bar{0},\chi}=\{ x-\chi(x)\mid  x \in \mathfrak{m}_{\bar{0}} \}. $$
The finite W-algebra $\mathcal{W}_0$ and W-superalgebra $\mathcal{W}$ are defined as follows:
\begin{equation} \label{equ2.1}
	\mathcal{W}_0=(\mathcal{U}_0/\mathcal{U}_0\mathfrak{m}_{\bar{0},\chi})^{\ad (\mathfrak{m}_{\bar{0}})}  \mbox{\quad and \quad} \mathcal{W}=(\mathcal{U}/\mathcal{U}\mathfrak{m}_{\chi})^{\ad(\mathfrak{m})}.
\end{equation}

Let us recall the Poisson geometric realization of finite W-(super)algebras in the sense of Losev.
Denote by $A_{0}$ (resp. $A$) the Poisson (resp. super) algebra
$S[\ggg_{\bar{0}}]$(resp. Let $S[\ggg]$) with the standard bracket $\{ ,\}$ given by $\{x,y\}=[x,y]$ for all $x,y \in \ggg_{\bar{0}}$ (resp. $\ggg$).
 $\hat{A}_{0}$ (resp. $\hat{A}$) be the completion of $A_{0}$ (resp. $A$) with respect to the point $\chi \in \ggg_{\bar{0}}^{*}$(resp. $\ggg$) and
  $\mathcal{U}_{\hbar,0}^{\wedge}$ (resp.  $\mathcal{U}_{\hbar}^{\wedge}$  )  be the formal quantization of $\hat{A}_{0}$ (resp. $\hat{A}$) given by $x\ast y -y\ast x=\hbar^{2}[x,y]$ for all $x,y \in \ggg_{\bar{0}}$. Equip all the above algebras with the Kazhdan $\mathbb{K}^{*}$-action arising from the  $\mathbb{Z}$-grading on $\ggg$ and $t\cdot \hbar=t\hbar$ for all $t \in \mathbb{K}^{*}$.

 Denote by $\omega$ the even symplectic form on $[f,\ggg]$ given by $\omega(x,y)=\chi([x,y])$. Here $f$ is the one in the $
 \mathfrak{sl}_2$-triple $\{e,f,h\}$ chosen in \S1.1. Let $V=V_{\bar{0}} \oplus V_{\bar{1}}$ be the superspace $ [f,\ggg]$ if $ \dim(\ggg(-1))$ is even. Otherwise, let
$V \subset [f,\ggg]$ be a superpace which has the standard basis $v_i$,  $i,j \in \{\pm1,\ldots, \pm (\dim([f,\ggg])-1)/2 \}$  with  $\omega(v_i,v_j)=\delta_{i,-j}$. We choose such  $V$ in the present paper following the definition of W-superalgebras given in \cite{Zh}.

For a superspace $V$ with an even symplectic form,  $\mathbf{A}_{\hbar}(V)$ denotes the corresponding Weyl superalgebra; see \cite[Example 1.5]{SX} for the definition. Specially, if $V$ is pure odd,  then the Weyl superalgebra $\mathbf{A}_{\hbar}(V)$  correspondences to the Clifford algebra $\mathrm{Cl}_{\hbar}(V)$.

It  is known \cite[\S2.3]{Lo11} that there is a $Q \times \mathbb{K}^{*}$-equivariant
$$\Phi_{0,\hbar}: \mathbf{A}_{\hbar}^{\wedge}(V_{\bar{0}}) \otimes \mathcal{W}_{0,\hbar}^{\wedge} \longrightarrow \mathcal{U}_{0,\hbar}^{\wedge} \footnote{Here and in Proposition \ref{prop2.1}, the tensor product is taken in the  category of complete, super $\mathbb{K}[[\hbar]]$-algebras. For simplicity of notation, similar abbreviations are  used frequently in the present paper. It is not hard to see their exact meaning from the context.} $$
isomorphism of associative algebras.  This can be extended as follows.

\begin{prop} \label{prop2.1}
\begin{itemize}
\item[(1)]
We have  a $Q \times \mathbb{K}^{*}$-equivariant
$$ \tilde{\Phi}_{\hbar}:  \mathbf{A}_{\hbar}^{\wedge}(V_{\bar{0}}) \otimes \mathcal{\tilde{W}}_{\hbar}^{\wedge} \longrightarrow \mathcal{U}_{\hbar}^{\wedge}$$
and a $\mathbb{K}^{*}$-equivariant isomorphism
$$\Phi_{1,\hbar}: \mathrm{Cl}_{\hbar}(V_{\bar{1}})\otimes \mathcal{W}_{\hbar}^{\wedge} \longrightarrow \mathcal{\tilde{W}}_{\hbar}$$
of associative algebras.
Finally this gives us a $\mathbb{K}^{*}$-equivariant isomorphism
$$ \Phi_{\hbar}:  \mathbf{A}_{\hbar}^{\wedge}(V) \otimes \mathcal{W}_{\hbar}^{\wedge} \longrightarrow \mathcal{U}_{\hbar}^{\wedge}$$
of associative algebras. Here  $\mathcal{\tilde{W}}_{\hbar}^{\wedge}$ is defined as the commutator of $\tilde{\Phi}_{\hbar}(V_{\bar{0}})$ in 
$\mathcal{U}_{\hbar}^{\wedge}$ and $\mathcal{W}_{\hbar}^{\wedge}$ is defined similarly.
\item[(2)] There are isomorphisms
$$ (\mathcal{\tilde{W}}_{\hbar}^{\wedge})_{\mathbb{K}^*-l.f}/(\hbar-1) \simeq \mathcal{\tilde{W}},\quad (\mathcal{W}_{0,\hbar}^{\wedge})_{\mathbb{K}^*-l.f}/(\hbar-1)\simeq \mathcal{W}_{0} \quad  \text{and} \quad
(\mathcal{W}_{\hbar}^{\wedge})_{\mathbb{K}^*-l.f}/(\hbar-1)\simeq \mathcal{W}$$
of associative algebra.
Here, for a vector space $V$ with a $\mathbb{K}^{*}$-action,  $(V)_{\mathbb{K}^*-l.f}$ denotes the sum of
all  finite dimensional $\mathbb{K}^{*}$-stable subspace of $V$.
\item[(3)] There is an embedding $ \mathfrak{q}:=\mathrm{Lie}(Q) \hookrightarrow  \mathcal{\tilde{W}}$ of Lie algebras such that
the adjoint action of $\mathfrak{q}$ coincides with  the differential of  the $Q$-action.
\end{itemize}
\end{prop}

\begin{proof}
(1)
Suppose that $V_{\bar{0}}$ has a basis $\{ v_{i} \}_{ 1 \leq |i| \leq l}$ with $\omega(v_i,v_j)=\delta_{i,-j}$.
The isomorphism  $\Phi_{0,\hbar}$ gives us a $Q$-equivariant embedding $\tilde{\Phi}_{\hbar}: V_{\bar{0}} \hookrightarrow \mathcal{U}_{\hbar}^{\wedge}$ with
$[\tilde{\Phi}_{\hbar}(v_i), \tilde{\Phi}_{\hbar}(v_j)]=\delta_{i,-j}\hbar$.  Now the isomorphism $\tilde{\Phi}_{\hbar}$ can be constructed as in the proof of \cite[Theorem 1.6]{SX}. For the construction of  $\Phi_{1,\hbar}$, see also Case 1 in the proof of \cite[Theorem 1.6]{SX}.
The isomorphism  $\Phi_{\hbar}$ can be constructed from the embedding $ \Phi_{\hbar}: V \hookrightarrow \mathcal{U}_{\hbar}^{\wedge}$ given by
$\Phi_{\hbar}|_{V_{\bar{0}}}=\tilde{\Phi}_{\hbar}$ and $\Phi_{\hbar}|_{V_{\bar{1}}}=\Phi_{1,\hbar}$.

(2) The second isomorphism was proved in \cite{Lo11}. The remaining statements follow by   similar arguments as in the proof of \cite[Theorem 3.8]{SX}.

(3) View $\mathcal{U}$ as a Harish-Chandra $\mathcal{U}_0$-bimodule and use \cite[\S2.5]{Lo11}.
\end{proof}

\begin{remark}
\begin{itemize}
	\item[(1)] In  Proposition \ref{prop2.1} above we do not claim that $\Phi_{\hbar}$ is $Q$-equivariant, although this is probably true.
	\item[(2)] Note that \eqref{equ2.1} can be interpreted as the Hamiltonian reduction of adjoint action of $\mathfrak{m}$ (resp. $\mathfrak{m}_{\bar{0}}$) on $\mathcal{U}$ (resp. $\mathcal{U}_0$). Similarly $\tilde{\mathcal{W}}$ can be viewed as the Hamiltonian reduction of 
	 $\mathfrak{m}_{\bar{0}}$-action on $\mathcal{U}$, namely $\tilde{\mathcal{W}}=(\mathcal{U}/\mathcal{U}\mathfrak{m}_{\bar{0},\chi})^{\ad(\mathfrak{m}_{\bar{0},\chi})}$. Moreover, there exists an
	odd commuting Lie superalgebra $\mathfrak{n} \subset \tilde{\mathcal{W}}$ such that $ \mathcal{W}=(\tilde{\mathcal{W}}/\tilde{\mathcal{W}}\mathfrak{n})^{\ad(\mathfrak{n})}$. Thus we may divide the reduction $(\mathcal{U}/\mathcal{U}\mathfrak{m}_{\chi})^{\ad(\mathfrak{m})}$ into two steps.  The algebra $\tilde{\mathcal{W}}$ 
	is obtained from the first one. Our setting can be viewed as an example of quantum super versions of the reduction by stages 
	in the classical symplectic geometry; see \cite{MMOPS}.  
	\end{itemize}
\end{remark}	

Proposition \ref{prop2.1} gives us the following  $Q \times \mathbb{K}^{*}$-equivariant version of  \cite[Theorem 4.1]{SX}. 
\begin{theorem}\label{them1}
\begin{itemize}
\item[(1)] We have a $Q \times \mathbb{K}^{*}$-equivariant  embedding  $\mathcal{W}_{0} \hookrightarrow  \tilde{\mathcal{W}}$ of associative  algebras. The latter is generated over the former by $\dim(\ggg_{\bar{1}})$ odd elements.
\item[(2)] Moreover we have an isomorphism
$$ \Phi_1 : \tilde{\mathcal{W}} \longrightarrow \mathrm{Cl}(V_{\bar{1}})\otimes_{\mathbb{K}} \mathcal{W}$$
 of associative algebras. Here $\mathrm{Cl}(V_{\bar{1}})$ is the Clifford algebra on the vector space $V_{\bar{1}}$
 with the symmetric bilinear form $\chi([\cdot,\cdot])$.
\end{itemize}
\end{theorem}

\begin{proof}
 The  proof is similar to the proof of  \cite[Theorem 4.1]{SX} and will be skipped.
 \end{proof}

Since it will be frequently used in later, it is helpful to recall the construction of  $\Phi_1$ in the following slightly general setting.

\begin{prop}\label{coro2.4}
	For a two-sided ideal  $ \tilde {\mathcal{I}}$ of  $\tilde{\mathcal{W}}$,    we have $\tilde {\mathcal{I}}=\mathrm{Cl}(V_{\bar{1}})\otimes_{\mathbb{K}} \mathcal{I}$. Here $\mathcal{I}$ is the two-sided 
ideal of $\mathcal{W}$  consisting of  elements anti-commuting with  $\mathrm{Cl}(V_{\bar{1}})$.
\end{prop}

\begin{proof}

By   Theorem   \ref{them1} (2)   there exist $x_1, \ldots, x_{\dim(V_{\bar{1}})} \in \tilde{\mathcal{W}}$  with 
$$x_i^2=1 \text {  and } x_ix_j=-x_jx_i \text{ for all  distinct $i , j  \in \{ 1, \ldots , \dim(V_{\bar{1}})\}$.}$$     
By  a quantum analog of \cite[Lemma 2.2(2)]{SX},  we have  that 
$ \tilde{\mathcal{I}}=\mathrm{Cl}(\mathbb{K}\langle  x_1 \rangle ) \otimes_{\mathbb{K}} \tilde{\mathcal{I}}_1$ as associative algebras.  Here  $\tilde{\mathcal{I}}_1$ denotes  the space anti-commuting with $x_1$.
Now the corollary follows by induction on $\dim(V_{\bar{1}})$.
\end{proof}

 \subsection{Equivalence of  $\mathcal{W}$-$\mathrm{Mod} $ and $ \mathcal{\tilde{W}}$-$\mathrm{Mod}$}
 
Let $\mathfrak{u}_{\bar{1}}$ be a Lagrangian of $V_{\bar{1}}$ and  $\mathfrak{u}^*_{\bar{1}}$ be its dual (given by the non-degenerate symmetric two form).
Note that  $V_{\bar{1}}=\mathfrak{u}_{\bar{1}} \oplus \mathfrak{u}^*_{\bar{1}}$.
View the exterior algebra 
$\bigwedge(\mathfrak{u}_{\bar{1}}^*)$ as a $ \mathrm{Cl}(V_{\bar{1}})$-module by 
$$ u\cdot x=ux \text{ and } v \cdot x=\omega(v,x)  \text{ for all $u,x \in \mathfrak{u}_{\bar{1}}^* $ and $v \in\mathfrak{u}_{\bar{1}}$ }. $$

The following proposition  establishes an explicit relation between the categories  
$\mathcal{W}$-Mod and $\mathcal{\tilde{W}}$-Mod. It relates to Proposition \ref{coro2.4}
via the bijective map $\mathrm{Irr}^{\mathrm{fin}}(\mathcal{\tilde{W}}) \rightarrow \Prim^{\fin}(\mathcal{\tilde{W}})$. 

\begin{prop} \label{prop2.5}
	For any $M \in \mathcal{\tilde{W}}$-$\mathrm{Mod}$, we have an isomorphism 
	$$ \bigwedge(\mathfrak{u}_{\bar{1}}^*) \otimes_{\mathbb{K}}  M' \rightarrow M: x \otimes m \mapsto x \cdot m $$ 
	of $\mathcal{\tilde{W}}$-modules. Here $M'$ is the annihilator of $\mathfrak{u}_{\bar{1}}$,
	which is naturally a $\mathcal{W}$-module and we view $\bigwedge(\mathfrak{u}_{\bar{1}}^*) \otimes_{\mathbb{K}}  M'$
	as a $\mathcal{\tilde{W}}$-module by the isomorphism  $\Phi_1$ in Theorem \ref{them1}.
	The functor $ \mathcal{\tilde{W}}\mathrm{\text{-}Mod}  \rightarrow \mathcal{W}\mathrm{\text{-}Mod}: M \mapsto M'$  
	is an equivalence of categories with the inverse $N \mapsto \bigwedge(\mathfrak{u}_{\bar{1}}^*) \otimes_{\mathbb{K}} N$.
\end{prop}

The proof is very similar to the proof of Proposition \ref{coro2.4} and  \cite[Lemma 2.2(2)]{SX}. 
\begin{proof}
	Let $x_1, \ldots,  x_{\dim(\mathfrak{u}_{\bar{1}})}$ be a basis of $\mathfrak{u}_{\bar{1}}$
	and $x_1^*, \ldots,  x_{\dim(\mathfrak{u}_{\bar{1}})}^*$ be the dual basis of $\mathfrak{u}_{\bar{1}}^*$ with $\omega(x_i,x_j^*)=\delta_{i,j}$. We claim that there is an  isomorphism  
	$$ \Psi_1 :\mathrm{Cl} (\mathbb{C}\langle x_1,x_1^*\rangle) \otimes_{\mathbb{K}} \mathcal{\tilde{W}}_1 \rightarrow \mathcal{\tilde{W}}$$ 
	of associative algebras. Here $\mathcal{\tilde{W}}_1$ is the super-commutator  of $x_1,x_1^*$ in $\mathcal{\tilde{W}}$ and the isomorphism is given by the multiplication in  $\mathcal{\tilde{W}}$. For any $y \in \mathcal{\tilde{W}} $, we have 
	
	$$\begin{array}{lll}
		&y=y-x_1[x_1^*,y]-x_1^*[x_1,y] \\
		&+x_1([x_1^*,y]-x_1^*[x_1,[x_1^*,y]]) +x_1x_1^*[x_1,[x_1^*,y]] \\
		&+x_1^*([x_1,y]-x_1[x_1^*,[x_1,y]]) +x_1x_1^*[x_1^*,[x_1,y]].
	\end{array}$$
	Therefore $\Psi_1$ is surjective. Suppose that 
	
	$$w_0 +x_1w_1+x_1^*w_2+x_1x_1^*w_3=0$$
	for some $w_i \in \mathcal{\tilde{W}}_1,i=0,1,2,3.$
	Applying the operator $[x_1,[x_1^*,\bullet]]$ on the both side we have $w_3=0$. By the same token, we have  $w_i=0$ for $i=0,1,2$.  So $\Phi_1$ is also injective. 
	Thus the claim follows.
	Now we prove the proposition for the pair $(\mathcal{\tilde{W}}_1,\mathcal{\tilde{W}})$.
	Namely there is an isomorphism 
	$$\Psi_{1,M}: \bigwedge(x_1^*) \otimes_{\mathbb{K}}  M'_1 \rightarrow M $$
	of $\mathcal{\tilde{W}}$-modules. Here the notation has a similar meaning as in the proposition. Indeed, for any $m \in M$, we have 
	$$m=m-x_1^*(x_1\cdot m)+x_1^*x_1\cdot m.$$
	Since $x_1\cdot m$ and $ m-x_1^*(x_1\cdot m) \in M'_1 $,   $\Psi_{1,M}$ is surjective. Similarly, we can check that 
	$\Psi_{1,M}$ is injective. Now the first statement follows by repeating the above procedure $\dim(\mathfrak{u}_{\bar{1}})$ times. The second statement is a direct consequence of the first one. 
	\end{proof}

\subsection{Maps $\bullet^{\dag}$ and $\bullet_{\dag}$}

We recall the constructions of maps $\bullet^{\dag}$ and $\bullet_{\dag}$  between $\mathfrak{id}(\mathcal{W})$
and $\mathfrak{id}(\mathcal{U})$ in \cite{SX} at  first. For  $\mathcal{I} \in \mathfrak{id}(\mathcal{W})$, we denote by  $\mathrm{R}_{\hbar}(\mathcal{I}) \subset \mathcal{W}_{\hbar}$  the Rees algebra
associated with $\mathcal{I}$ and by $\mathrm{R}_{\hbar}(\mathcal{I})^{\wedge} \subset \mathcal{W}_{\hbar}^{\wedge}$ the completion of $\mathrm{R}_{\hbar}(\mathcal{I})$ at $0$. Let $\mathbf{A}(\mathcal{I})^{\wedge}_{\hbar}=\mathbf{A}_{\hbar}(V)^{\wedge}\otimes \mathrm{R}_{\hbar}(\mathcal{I})^{\wedge}$ and set $\mathcal{I}^{\dag}=(\mathcal{U}_{\hbar} \cap \Phi_{\hbar}(\mathbf{A}(\mathcal{I})^{\wedge}_{\hbar}))/(\hbar-1)$. For an ideal $\mathcal{J} \in \mathfrak{id}(\mathcal{U})$, $\bar{\mathcal{J}}_{\hbar}$  stands for the closure of $\mathrm{R}_{\hbar}(\mathcal{J})$ in $\mathcal{U}^{\wedge_{\chi}}_{\hbar}$. Define $\mathcal{J}_{\dag}$ to be the unique (by \cite[ Proposition 3.4(3)]{SX}) ideal in $\mathcal{W}$  such that
$\mathrm{R}_{\hbar}(\mathcal{J}_{\dag})=\Phi_{\hbar}^{-1}(\bar{\mathcal{J}}_{\hbar}) \cap \mathrm{R}_{\hbar}(\mathcal{W}).$

A $\mathfrak{g}_{\bar{0}}$-bimodule $M$ is said to be Harish-Chandra(HC) bimodule,  if $M$ is finitely generated and the adjoint action of $\mathfrak{g}$ on $M$ is locally finite. For any two-sided ideal $\mathcal{J}\subset \mathcal{U}$ (resp. $\mathcal{I}\subset \tilde{\mathcal{W}}$),  $\mathcal{J}_{\tilde{\dag}}$ (resp. $\mathcal{I}^{\tilde{\dag}}$) denotes the image of $\mathcal{J}$
under the functor $\bullet_{\dag}$ (resp. $\bullet^{\dag}$) in \cite[\S3]{Lo11}.  Here we view $\mathcal{J}$ and $\mathcal{I}$   as  HC-bimodules over $\mathfrak{g}_{\bar{0}}$ and $\mathcal{W}_{0}$ respectively.
\begin{lemma} \label{lem2.4}
We have that $(\mathrm{Cl(V_{\bar{1}})} \otimes_{\mathbb{K}} \mathcal{I})^{\tilde{\dag}}=\mathcal{I}^{\dag}$ and $\mathcal{I}_{\tilde{\dag}}=\mathrm{Cl}(V_{\bar{1}})\otimes_{\mathbb{K}} \mathcal{I}_{\dag}$.
\end{lemma}

\begin{proof}
The  $\mathbb{K}^{*}$-action (see the paragraph before Lemma 3.3.3, \cite{Lo11} ) defining  HC $\mathcal{U}_{0}$-bimodule $\tilde{\mathcal{I}}^{\tilde{\dag}}$ is 
given by $t\cdot x=t^{-2}x$ for all $x \in \ggg$ and $t \in \mathbb{K}^{*}$.
So $\mathcal{U}_{\hbar} \cap \Phi_{\hbar}(\mathbf{A}(\mathcal{I})^{\wedge}_{\hbar})$ 
coincides with the  $\mathbb{K}^{*}$-local finite part of $\Phi_{\hbar}(\mathbf{A}(\mathcal{I}))^{\wedge}_{\hbar}$. Thus the lemma follows.
For a similar fact in the even case, see \cite[Remark 3.4.4]{Lo11}. 
\end{proof}

\subsection{ Properties of  $\bullet^{\dag}$ and $\bullet_{\dag}$}

For an associative algebra $\mathcal{A}$,  $\mathrm{GK}\dim(\mathcal{A})$ denotes the Gelfand-Kirillov dimension of $\mathcal{A}$ (for the definition, see\cite{KL}).
The \textit{associated variety} $\mathbf{V}(\mathcal{J} )$ of a two-sided ideal $\mathcal{J} \in \mathfrak{id}(\mathcal{U})$, is defined to be the associated variety $\mathbf{V}(\mathcal{J}_{0})$ of $ \mathcal{J}_{0}=\mathcal{J} \cap \mathcal{U}_0$. We say that $\mathcal{J}$ is\textit{ supported} on $\mathbf{V}(\mathcal{J})$ in this case.

\begin{lemma}\label{legkdim}
For any two-sided ideal of $\mathcal{J}\subset \mathcal{U}$, we have
$$\mathrm{GK}\dim(\mathcal{U}/\mathcal{J})=\mathrm{GK}\dim(\mathcal{U}_0/\mathcal{J}_{0} )=\dim (\mathbf{V}(\mathcal{J})).$$
\end{lemma}
\begin{proof}
Note that we have the natural embedding $ \mathcal{U}_0/\mathcal{J}_{0}  \hookrightarrow \mathcal{U}/\mathcal{J}$. The first equality follows from the definition of Gelfand-Kirillov dimension (see \cite[pp.14 Definition]{KL} and the remark following it) and the PBW basis theorem.
The second equality follows from\cite[Corollary 5.4]{BK}.
\end{proof}

The following proposition is a super generalization of \cite[Theorem 1.2.2 (\rmnum{7})]{Lo10} in a special case.
\begin{prop}\label{finitefiberprop}
For any $ \mathcal{J} \in \mathrm{Prim}_{\mathbb{O}}(\mathcal{U})$,   $\{  \mathcal{I}  \in \mathfrak{id} (\mathcal{W}) \mid  \text{ $\mathcal{I} $ is prime, $ \mathcal{I}^{\dag}=\mathcal{J}$}\}$ is exactly the set consisting of the minimal prime ideals containing $\mathcal{J}_{\dag}$.
\end{prop}

\begin{proof}
Suppose that $\mathcal{I}$ is a prime ideal of $\mathcal{W}$ with $\mathcal{I}^{\dag}=\mathcal{J}$.   \cite[Proposition 4.5]{SX} implies that  $\mathcal{J}_{\dag} \subset \mathcal{I} $.  So $\mathcal{I}$ has finite codimension in  $\mathcal{W}$.  Hence $\mathcal{I}$ is minimal by  \cite[Corollary 3.6]{BK}.
Let $\mathcal{I} \subset \mathcal{W}$ be a minimal prime ideal  with $ \mathcal{J}_{\dag}\subset \mathcal{I}$. According to \cite[Proposition 4.6]{SX}, $\mathcal{J}_\dag$ has finite codimension in  $\mathcal{W}$. Thus we can see that  $ \tilde{\mathcal{I}}=\mathrm{Cl}(V_{\bar{1}}) \otimes_{\mathbb{K}} \mathcal{I}$ has finite codimension in $\tilde{\mathcal{W}}$. Whence $\tilde{\mathcal{I}}_{0}=\mathcal{W}_{0} \cap \tilde{\mathcal{I}}$ has  finite  codimension in $\mathcal{W}_{0}$.  Since $\mathcal{I}^{\dag} \cap \mathcal{U}_{0}=(\tilde{\mathcal{I}}_{0})^{\tilde{\dag}}$, we obtain  that $\mathcal{I}^{\dag}$ is supported on $G_{\bar{0}}\cdot \chi$ by the proof of \cite[ Theorem 1.2.2 (\rmnum{7})]{Lo10}. Thus  Lemma \ref{legkdim}  in conjunction with \cite[ Corollary 3.6]{BK} yields $\mathcal{I}^{\dag}=\mathcal{J}$.
\end{proof}

Let $\sigma$ be the automorphism of superalgebra $\mathcal{A}=\mathcal{A}_{\bar{0}}+\mathcal{A}_{\bar{1}}$ given by $\sigma(x)=x_0-x_1$ for any $x=x_0+x_1$ in $\mathcal{A}$. An ideal of  $\mathcal{A}$ is $\mathbb{Z}_2$-graded 
if and only if it is invariant under $\sigma$. We have the following relation between primitive 
and graded primitive ideals of $\mathcal{A}$.
\begin{lemma}$($\cite[Lemma 7.6.3]{Mu12}$)$ \label{graded prime}
For any graded 	primitive ideal $\mathcal{I}^{'}$ of $\mathcal{A}$, there exists a primitive ideal $\mathcal{I} \subset \mathcal{A} $
such that $\mathcal{I}^{'}=\mathcal{I} \cap \sigma(\mathcal{I})$.
\end{lemma}

\subsection{Proof of main result Theorem \ref{mainthm}}

We prove the theorem  by a similar argument as in the proof of \cite[Conjecture 1.2.1]{Lo11}. Indeed,  given  $\mathcal{J} \in \mathrm{Prim}_{\mathbb{O}}(\mathcal{U})$, let $\mathcal{I}_{1}, \ldots, \mathcal{I}_{l}$ be the minimal prime ideals containing $\mathcal{J}_{\dag}$. Since $ \mathrm{Cl}(V_{\bar{1}})\otimes_{\mathbb{K}}\mathcal{I}_{1}$ is stable under $Q^{\circ}$,  $ \bigcap_{\gamma \in C_e} \gamma (\mathrm{Cl}(V_{\bar{1}}) \otimes_{\mathbb{K}} \mathcal{I}_{1})$ is $Q$-stable.  Set $\mathcal{J}^{1}= (\bigcap_{\gamma \in C_e} \gamma (\mathrm{Cl}(V_{\bar{1}}) \otimes_{\mathbb{K}} \mathcal{I}_{1}))^{\tilde{\dag}}$, then by \cite[Theorem 4.1.1]{Lo11} we have   $(\mathcal{J}^{1})_{\tilde{\dag}}= \bigcap_{\gamma \in C_e} \gamma (\mathrm{Cl}(V_{\bar{1}}) \otimes_{\mathbb{K}} \mathcal{I}_{1})$.
Thus  $\mathcal{J} =(\mathcal{I}_1)^{\dag} \supset \mathcal{J}^{1}  \supset  \mathcal{J} $  (the first equality follows from Lemma \ref{legkdim} and  \cite[Corollary 3.6]{BK}). 
Hence $\mathcal{J}_{\tilde{\dag}}=\bigcap_{\gamma \in C_e} \gamma (\mathrm{Cl}(V_{\bar{1}}) \otimes_{\mathbb{K}} \mathcal{I}_{1})$.  We obtain  that $\gamma (\mathrm{Cl}(V_{\bar{1}}) \otimes_{\mathbb{K}} \mathcal{I}_{1})=\mathrm{Cl}(V_{\bar{1}})\otimes_{\mathbb{K}} \mathcal{I}_{\gamma(1)}$ for some $\gamma(1) \in \{1,\ldots, l \}$ by  \cite[Proposition 3.1.10]{Di} and Corollary \ref{coro2.4}.  Thus we have
$\mathcal{I}=\bigcap_{\gamma \in C_e}\mathcal{I}_{\gamma(1)}$ by \cite[Proposition 3.1.10]{Di} and Lemma \ref{lem2.4}.
Now the proof is completed by Proposition \ref{finitefiberprop}.

In the $\mathbb{Z}_2$-graded case,  the automorphism given by 
$g \in Q$ commutes with $\sigma$. Thus the second statement follows from the first one and Lemma \ref{graded prime}. $\hfill\Box$
                                             
 \subsection{ Finite dimensional representations of $\mathcal{\tilde{W}}$ }\label{sec2.5}                                                                                                                                            
 
  Now we point out the role of Theorem \ref{mainthm}  in describing  $\mathrm{Irr}^{\mathrm{fin}}(\mathcal{\tilde{W}})$.  We mentioned earlier that the map  $$ \mathrm{Irr}^{\mathrm{fin}}(\mathcal{\tilde{W}}) \rightarrow \Prim^{\fin}(\mathcal{\tilde{W}}): M \mapsto \Ann(M) $$                                                                                             
 is bijective.  Given $\mathcal{I} \in \Prim^{\fin}(\mathcal{\tilde{W}})$,  $\mathcal{\tilde{W}}/\mathcal{I}$ is isomorphic to $\End(M)$ for some finite dimensional vector space $M$ over $\mathbb{K}$ by a well-known fact for general finite dimensional simple algebras. The inverse of the above map is given by $I \mapsto M$. 
 By Lemma \ref{graded prime} there is a similar bijection in the $\mathbb{Z}_2$-graded case.
 
Now let $M \in \mathrm{Irr}^{\mathrm{fin}}(\mathcal{\tilde{W}})$ and $\mathcal{I}=\Ann(M)$. If  $g \in C_e=Q/Q^{\circ}$ and  $g' \in Q$ is a representative of $g$,   $ ^gM$ denotes the twist of $M$ by the algebra automorphism $g'$ of $\mathcal{\tilde{W}}$. Obviously, the annihilator of $ ^gM$  is $g\cdot \mathcal{I}$. Thus Theorem \ref{mainthm} is 
equivalent to saying that $ \{  ^gM | g \in C_e \}$ equals the set of modules in $\mathrm{Irr}^{\mathrm{fin}}(\mathcal{\tilde{W}})$ which are
annihilated by $(\mathcal{I}^{\tilde{\dag}})_{\tilde{\dag}}$.
 
\subsection{ In the special  case:  $C_e=1$} \label{2.4}

For a basic Lie superalgebra $\ggg=\ggg_{\bar{0}}+\ggg_{\bar{1}}$ of type \Rmnum{1}, Lezter established a bijection $\nu :  \mathrm{Prim} (\mathcal{U}_{0})\rightarrow  \mathrm{Prim} (\mathcal{U})$.
It follows from the construction that $\nu$ restricts to a  bijection between
$  \mathrm{Prim}_{\mathbb{O}} (\mathcal{U}_{0})$ and $  \mathrm{ Prim} _{\mathbb{O}}(\mathcal{U})$.
So  we can describe  $\mathrm{Irr}^{\mathrm{fin}}(\mathcal{W})$ when $C_e$ is trivial.  We know that  the finite group  $C_e$ is trivial when $\ggg$ is of  type $A(m|n)$  or  $e$ is a principal nilpotent element in the type 
$C(n)$ Lie superalgebras.  In the case of  $\ggg=\mathfrak{osp}(1,2n)$, a  description of  $\mathrm{Prim}(\mathcal{U})$   is given in  \cite[Theorem A, B ]{Mu97}.  The poset structure describing  $\mathrm{Prim}(\mathcal{U})$   is  exactly the same as  that  of $\mathrm{Prim} (\mathcal{U}_{0})$.   It is straightforward to check that  $\widehat{L}(\lambda)$ is supported on $\bar{\mathbb{O}}$ if and only if so is $L(\lambda)$. Thus we show that Theorem \ref{mainthm} gives a description  of $\mathrm{Irr}^{\mathrm{fin}}(\mathcal{W})$ provided $C_{e}=1$.


\section{ Graded irreducible representations}

From now on, let $\ggg$ be a basic Lie superalgebra of type \Rmnum{1}. 
The most essential feature is that they  admit a $\mathbb{Z}_2$-compatible $\mathbb{Z}$-grading
$$\ggg=\ggg_{-1} \oplus \ggg_0 \oplus \ggg_{1}$$ of Lie superalgebras.  Here the term $\mathbb{Z}_2$-compatible means $\ggg_{-1}\oplus\ggg_{1}=\ggg_{\bar{1}}$ and $ \ggg_0=\ggg_{\bar{0}}$.
For a $\ggg_{\bar{0}}$-module $V$, view it as a  $\ggg_0+\ggg_{1}$-module with the trivial $\ggg_{1}$ action and define 
$K(V)=\mathrm{ind}_{\ggg_0+\ggg_{1}}^{\ggg} V$.
 We refer to $K(\bullet)$ as the  Kac functor from the category of $\ggg_{\bar{0}}$-modules to that of $\ggg$-supermodules. 
 The main result of \cite{CM} states  that, for any simple $\ggg_0$-supermodule $V$, the Kac module $K(V)$ has a unique simple $\mathbb{Z}_2$-graded quotient  $\widehat{V}$, and the map $V \mapsto \widehat{V}$ induces  a bijection between the set of isomorphism classes of  simple $\ggg_0$-modules and of simple $\ggg$-supermodules.  It is well known that the above map sends the highest weight simple $\mathcal{U}_{0}$-module $L(\lambda)$ to the highest weight simple $\mathcal{U}$-module $\widehat{L}(\lambda)$.  Now we could give a classification of simple $\tilde{\mathcal{W}}$-supermodules (hence of $\mathcal{W}$-supermodules) via the Kac equivalence and  Skryabin's equivalence.

A $\ggg$-supermodule $M$ is called Whittaker if $\mathfrak{m}_{\bar{0},\chi}$ acts on it as a locally nilpotent endomorphism. It is 
easy to check that $\widetilde{\mathrm{Wh}}(M):=M^{\mathfrak{m}_{\bar{0},\chi}}$ is a $\tilde{\mathcal{W}}$-supermodule.
Let $\tilde{Q}_{\chi}$ be the left $\mathcal{U}$-module $\mathcal{U}/\mathcal{U}\mathfrak{m}_{\bar{0},\chi}$. It also has a right 
$\tilde{\mathcal{W}}$-supermodule structure. For any $\tilde{\mathcal{W}}$-supermodule $N$, $\tilde{Q}_{\chi}\otimes_{\tilde{\mathcal{W}}}N$ 
is a left $\mathcal{U}$-supermodule. Let  $Q_{0,\chi}=\mathcal{U}_0/\mathcal{U}_0\mathfrak{m}_{\bar{0},\chi}$ be the $(\mathcal{U}_0,\mathcal{W}_0)$-bimodule defined similarly. We have the following Skryabin's equivalence for $\tilde{\mathcal{W}}$.
\begin{theorem}\label{Skrabin equi thm}
	The functor $\widetilde{\mathrm{Wh}}$ and $Q_\chi \otimes_{\tilde{\mathcal{W}}} \bullet $ are mutual quasi-equivalences between the categories of
	$\tilde{\mathcal{W}}$-supermodules and  of Whittaker $\mathcal{U}$-supermodules. 
 For any $\tilde{\mathcal{W}}$-supermodule $N$,   $Q_{0,\chi} \otimes_{\mathcal{W}_0} N $ also has a $\mathcal{U}$-supermodule  structure, which is isomorphic to $\tilde{Q}_{\chi} \otimes_{\tilde{\mathcal{W}}}N$. 
\end{theorem}
The second statement is very useful in our study of representations of $\tilde{\mathcal{W}}$. It enables us to use results on $\mathcal{W}_{0}$. We may prove the theorem by a similar argument in the  W-algebra cases; see \cite{Lo11} or  \cite{SX} in W-superalgebra cases. Here we provide a sketch to prove it.
\begin{proof}
	
Let $\mathbf{A}_{V_{\bar{0}}}(\tilde{\mathcal{W}})=\mathbf{A}(V_{\bar{0}}) \otimes_{\mathbb{K}} \tilde{\mathcal{W}}$. 
We claim that there is an isomorphism $\mathcal{U}^{\wedge}_{\mathfrak{m}_{\bar{0},\chi}} \rightarrow \mathbf(A)(\tilde{\mathcal{W}})^{\wedge}_{\mathfrak{m}_{\bar{0}}}$ of topological algebras, 
 where $\mathcal{U}^{\wedge}_{\mathfrak{m}_{\bar{0},\chi}}$(resp. $\mathbf{A}_{V_{\bar{0}}}(\tilde{\mathcal{W}})^{\wedge}_{\mathfrak{m}_{\bar{0}}}$) is the completion of $\mathcal{U}$(resp. $\mathbf{A}_{V_{\bar{0}}}(\tilde{\mathcal{W}})$) with respect to the nilpotent Lie subalgebra $\mathfrak{m}_{\bar{0},\chi} \subset \mathcal{U} $ (resp. commutative subalgebra $\mathfrak{m}_{\bar{0}}$). This is an analog of  \cite[Theroem 1.2.1]{Lo10}  for $\mathcal{W}_{0}$, which states that  $(\mathcal{U}_0)^{\wedge}_{\mathfrak{m}_{\bar{0},\chi}}$ is isomorphic to  $(\mathbf{A}(V_{\bar{0}}) \otimes_{\mathbb{K}} \mathcal{W}_0)^{\wedge}_{\mathfrak{m}_{\bar{0}}} $ as topological algebras. Our claim can be proved by the similar arguments therein.

 View  $Q_{0,\chi}$ as an $\mathbf{A}(V_{\bar{0}}) \otimes_{\mathbb{K}} \mathcal{W}_0$-module via the above second isomorphism,  then we have  
 $Q_{0,\chi}=\mathbb{K}[\mathfrak{m}_{\bar{0}}] \otimes_{\mathbb{K}} \mathcal{W}_0$ as ($\mathbf{A}(V_{\bar{0}}) \otimes_{\mathbb{K}} \mathcal{W}_0,\mathcal{W}_{0})$-bimodules; see \cite[pp.52]{Lo10}). Similarly we have  $\tilde{Q}_{\chi}=\mathbb{K}[\mathfrak{m}_{\bar{0}}] \otimes_{\mathbb{K}} \tilde{\mathcal{W}}$ as $(\mathbf{A}_{V_{\bar{0}}}(\tilde{\mathcal{W}}),\tilde{\mathcal{W}})$-bimodules. Therefore   
 $$Q_{0,\chi} \otimes_{\mathcal{W}_0}N= (\mathbb{K}[\mathfrak{m}_{\bar{0}}] \otimes_{\mathbb{K}} \mathcal{W}_0) \otimes_{\mathcal{W}_0} N=\mathbb{K}[\mathfrak{m}_{\bar{0}}] \otimes_{\mathbb{K}} N $$ has an
 $\mathbf{A}_{V_{\bar{0}}}(\tilde{\mathcal{W}})$-supermodule structure. 
 Hence it is a 
Whittaker $\mathcal{U}$-supermodule via the homomorphism  $\mathcal{U} \hookrightarrow  \mathcal{U}^{\wedge}_{\mathfrak{m}_{\bar{0},\chi}} \rightarrow \mathbf{A}_{V_{\bar{0}}}(\tilde{\mathcal{W}})^{\wedge}_{\mathfrak{m}_{\bar{0}}}$. Repeating the argument of \cite[Theorem 4.1]{SX}, the theorem follows.  
 \end{proof} 

\begin{theorem}\label{Classification Thm}
	The sets 	
	$\mathrm{Irr}(\mathcal{W}_0)$, $\gr.\mathrm{Irr}(\tilde{\mathcal{W}})$ and $\gr.\mathrm{Irr}(\mathcal{W})$ are bijective with each other. Any simple $\tilde{\mathcal{W}}$-supermodule, or equivalently simple $\mathcal{W}$-supermodule is $\mathbb{Z}$-gradable. 
\end{theorem}

\begin{proof}
	Obviously, the Kac functor 
	maps the Whittaker $\ggg_{\bar{0}}$-modules to the Whittaker $\ggg$-supermodules. 
	According to  Theorem \ref{Skrabin equi thm}, we have that 
	the map $N \mapsto \widehat{N}:= \widetilde{\mathrm{Wh}} (\widehat{Q_{\chi,0} \otimes_{\mathcal{W}_0} N})$ is a bijection between $\mathrm{Irr}(\mathcal{W}_0)$ and $\gr.\mathrm{Irr}(\tilde{\mathcal{W}})$. Since $\mathfrak{m}_{\bar{0},\chi} \subset \mathcal{U}$ is $\mathbb{Z}$-homogeneous, the second statement follows from the fact that any simple $\ggg$-supermodule is $\mathbb{Z}$-gradable; see the proof of  \cite[Theorem 4.1]{CM}.
\end{proof}

\section{Character formula} \label{sec cher for}

\subsection{Triangular decomposition for $\mathcal{\tilde{W}}$}

Let $\mathcal{U}_{+}$(resp. $\mathcal{U}_{-}$) be the universal enveloping algebra of 
$\ggg_{0}+\ggg_{1}$ (resp. $\ggg_{0}+\ggg_{-1}$). Define their completion $(\mathcal{U}_{+})_{\hbar}^{\wedge}$ and $(\mathcal{U}_{-})_{\hbar}^{\wedge}$ similarly to
 $\mathcal{U}_{\hbar}^{\wedge}$.  The restrictions of $\tilde{\Phi}_\hbar$ to $(\mathcal{U}_{+})_{\hbar}^{\wedge}$ and $(\mathcal{U}_{-})_{\hbar}^{\wedge}$ give the following isomorphisms 
 $$\tilde{\Phi}_{\hbar}^+:  \mathbf{A}_{\hbar}^{\wedge}(V_{\bar{0}}) \otimes \mathcal{\tilde{W}}_{+,\hbar}^{\wedge} \longrightarrow(\mathcal{U}_{+})_{\hbar}^{\wedge}
 \text{ and }
\tilde{\Phi}_{\hbar}^-:  \mathbf{A}_{\hbar}^{\wedge}(V_{\bar{0}}) \otimes \mathcal{\tilde{W}}_{-,\hbar}^{\wedge} \longrightarrow(\mathcal{U}_{-})_{\hbar}^{\wedge}$$
of associative algebras. Here $\mathcal{\tilde{W}}_{+,\hbar}^{\wedge}$ and $\mathcal{\tilde{W}}_{-,\hbar}^{\wedge}$ are defined similarly to $\mathcal{\tilde{W}}_{\hbar}^{\wedge}$ in Proposition \ref{prop2.1}. 
Define $\mathcal{\tilde{W}}_{-}:=(\mathcal{\tilde{W}}_{-,\hbar}^{\wedge})_{\mathbb{K}^*-l.f}/(\hbar-1)$
and $\mathcal{\tilde{W}}_{+}:=(\mathcal{\tilde{W}}_{+,\hbar}^{\wedge})_{\mathbb{K}^*-l.f}/(\hbar-1)$.
They can be viewed as the W-superalgebras from $(\ggg_{-1}+\ggg_0,e)$ and  $(\ggg_{0}+\ggg_1,e)$.

Equip  $\mathcal{U}_{\hbar}$ a $\mathbb{Z}$-grading such that the subspace $\mathcal{U}$ has the natural grading from $\ggg$ and $\hbar$ has  the grading 
$0$. The isomorphism $\tilde{\Phi}_{\hbar}$  preserves the
$\mathbb{Z}$-grading by construction. Hence there is a $\mathbb{Z}$-grading $\mathcal{\tilde{W}}=\bigoplus_{i \in \mathbb{Z} } \mathcal{\tilde{W}}$ inherited from the one on $\mathcal{U}$, and $\mathcal{\tilde{W}}_{-}$ and $\mathcal{\tilde{W}}_{+}$ are $\mathbb{Z}$-graded subalgebras 
of $\mathcal{\tilde{W}}$.  

\begin{prop}\label{Triangular Decom}
	\begin{itemize}	
\item[(1)] There exist $\mathbb{Z}$-homogeneous odd elements $x^{-}_1,\ldots,x^{-}_k  \in \mathcal{\tilde{W}}_{-}$, $x^{+}_1,\ldots,x^{+}_k  \in \mathcal{\tilde{W}}_{+}$ and
 $x_1,\ldots,x_l  \in \mathcal{W}_0$ 
 such that they form a PBW basis of $\mathcal{\tilde{W}}$ in the super sense.
Where $k=\dim(\ggg_{-1})=\dim(\ggg_{1})$ and $l=\dim((\ggg_{\bar{0}})_e)$. We emphasize that $\mathcal{W}_0$
is the ordinary finite W-algebra from  $(\ggg_{\bar{0}},e)$.
	 
\item[(2)]
Let $\mathcal{\tilde{W}}_{-}^{\#}$ (resp. $\mathcal{\tilde{W}}_{+}^{\#}$) be the vector space of  exterior algebra generated by $x^{-}_1,\ldots,x^{-}_k$
(resp. $x^{+}_1,\ldots,x^{+}_k$). There  are isomorphism of vector spaces
\begin{equation}\label{trian decom}
\mathcal{\tilde{W}}\simeq \mathcal{\tilde{W}}_{+}^{\#}\otimes_{\mathbb{K}}\mathcal{W}_0\otimes_{\mathbb{K}}  \mathcal{\tilde{W}}_{-}^{\#},
\mathcal{\tilde{W}}_{+}\simeq \mathcal{W}_0 \otimes_{\mathbb{K}} \mathcal{\tilde{W}}_{+}^{\#},  \mathcal{\tilde{W}}_{-}\simeq\mathcal{W}_0 \otimes_{\mathbb{K}} \mathcal{\tilde{W}}_{-}^{\#}
\end{equation}
given by the multiplication of $\mathcal{\tilde{W}}$.   
\item[(3)]  For any irreducible $\mathcal{W}_0$-module $N$, view it as a $\mathcal{\tilde{W}}_{+}$-module via the quotient $\mathcal{\tilde{W}}_{+} \twoheadrightarrow \mathcal{W}_0$ modulo the two-sided ideal generated by elements with positive $\mathbb{Z}$-grading (or by the image of $\mathcal{\tilde{W}}_{+}^{\#}$ equivalently). Then the Verma module $\Delta^{K}_{\mathcal{\tilde{W}}}(N):=  \mathcal{\tilde{W}} \otimes_{\mathcal{\tilde{W}}_{+}} N $ has a unique simple quotient $L^{K}_{\mathcal{\tilde{W}}}(N)$. The map $ \mathrm{Irr}^{\mathrm{fin}}(\mathcal{W}_0) \rightarrow \gr.\mathrm{Irr}^{\mathrm{fin}}( \tilde{\mathcal{W}}): N \mapsto L^{K}_{\mathcal{\tilde{W}}}(N)$ is bijective and $C_e$-equivariant.
\end{itemize}
\end{prop}

\begin{proof} 
Statement (1) follows from a similar argument as the proof of existence of PBW  basis for $\mathcal{W}_0$
in \cite{Lo10} or for $\mathcal{W}$ in \cite{SX}. Let $\tilde{S}_e = (\ggg_{\bar{0}})_e\oplus \ggg_{\bar{1}} $ and choose odd elements $'x^{-}_1,\ldots,'x^{-}_k  \in \ggg_{-1} $,  $'x^{+}_1,\ldots,'x^{+}_k  \in \ggg_{1}$ and
$'x_1,\ldots,'x_l  \in (\ggg_{\bar{0}})_e $ 
such that they form a basis of the vector space $\tilde{S}_e$. The procedure in 
 \cite[(2.3)]{SX} shows that $(x^{-}_1)_{\hbar}:='x^{-}_1+ \square $ is in  $\tilde{\Phi}_{\hbar}(\tilde{\mathcal{W}}^{\wedge}_{\hbar})$. 
Here  $\square$ denotes the higher order correcting term obtained in there. We can 
construct  $(x^{\pm}_i)_\hbar$ for $i=2,\dots, k$ and  $(x_i)_\hbar$ for $i=1,\dots, l$ similarly.
Since $\tilde{\mathcal{W}}^{\wedge}_{\hbar}/(\hbar)=S[[\tilde{S}_e]]$,  these elements generate $\tilde{\mathcal{W}}^{\wedge}_{\hbar}$ as $\mathbb{K}[[\hbar]]$-algebra.
They also lie in $(\tilde{\mathcal{W}}^{\wedge}_{\hbar})_{\mathbb{K}^*-l.f}$, since they are $\mathbb{K}^*$-homogeneous. We 
can take the PBW basis  as their image under the quotient map $(\tilde{\mathcal{W}}^{\wedge}_{\hbar})_{\mathbb{K}^*-l.f} \rightarrow \tilde{\mathcal{W}} $ 
given by specializing $\hbar$ to $1$.

Claim (2) follows directly from (1). 

Let $M$ be a $\mathbb{Z}_2$-graded simple quotient of $\Delta^{K}_{\mathcal{\tilde{W}}}(N)$ and  $\pi$ be the quotient homomorphism. By Theorem \ref{Classification Thm}  we may assume  $M$ has a $\mathbb{Z}$-grading with top degree $0$.  We claim  that $\pi$ has to be a $\mathbb{Z}$-graded  homomorphism.  
Otherwise,  for a non-zero $x \in N$, we may write $\pi(x)=\sum_{i=1}^n y_i$ for $\mathbb{Z}$-homogeneous $y_i \in M$, $i=1,2, \ldots n>1$. Suppose $\gr(y_1)=d<0$.  Since $\mathcal{\tilde{W}}_{+}^{\#}\cdot y_1=0$,  submodule $\mathcal{\tilde{W}}\cdot y_1$ has top degree $d<0$, so it is a proper sub-supermodule of simple supermodule $M$, a contradiction. Thus we have that any  maximal sub-supermodule of $\Delta^{K}_{\mathcal{\tilde{W}}}(N)$ is a $\mathbb{Z}$-graded submodule. Consequently  the sum of all the proper maximal sub-supermodules of  $\Delta^{K}_{\mathcal{\tilde{W}}}(N)$ is the unique proper maximal sub-supermodule. 
For any $g \in C_e$,  it is clear that  $ ^g L_{\mathcal{\tilde{W}}}^{K}(N)= L_{\mathcal{\tilde{W}}}^{K}( ^gN)$.
The  claim (3) follows.
\end{proof}

 The following corollary combined with the main result of \cite{Lo12}, gives us a complete classification of $\gr.\mathrm{Prim}(\tilde{\mathcal{W}})$ in the  type A case.
 \begin{corollary}\label{Coro4.2} 
	For a basic type \Rmnum{1} Lie superalgebra $\ggg$,  the sets 
	$$ \mathrm{Prim}(\mathcal{W}_0),\ \gr.\mathrm{Prim}(\tilde{\mathcal{W}}) \ \mbox{and} \  \gr.\mathrm{Prim}(\mathcal{W})$$
	are  bijective  with each other.	
\end{corollary}

\begin{proof}
We decent Letzter's bijection $\nu:\mathrm{Prim}(\mathcal{U}_0) \rightarrow  \mathrm{Prim}(\mathcal{U})$  to
 $$\nu_{\mathcal{\tilde{W}}}: \mathrm{Prim}(\mathcal{W}_0) \rightarrow  \gr.\mathrm{Prim}(\mathcal{\tilde{W}}).$$
For any $\mathcal{I} \in \gr.\mathrm{Prim}(\mathcal{\tilde{W}})$,
let $\hat{\mathcal{I}}$ be the preimage of $\mathcal{I}$ under the quotient
$\mathcal{\tilde{W}}_{+} \twoheadrightarrow \mathcal{W}_0$. We claim that 
there is a unique primitive ideal of $ \mathcal{\tilde{W}}$ lying over $\hat{\mathcal{I}}$.
We define $\nu_{\mathcal{\tilde{W}}}(\mathcal{I})$ to be this primitive ideal.
The claim, hence the corollary,  can be proved by  repeating the proof of Letzter's theorem \cite[Theorem 15.2.5]{Mu12} almost  word by word. In fact, we only need to replace  $R$ and $Q$ therein by $\mathcal{\tilde{W}}_{+}$ and $\mathcal{\tilde{W}}$.
 Proposition \ref{Triangular Decom} is used to verify the conditions of  \cite[Lemma 7.6.12]{Mu12}.
\end{proof}

\subsection{ Recall: the generalized Soergel functor $\mathbb{V}$ in the even theory}\label{sec4.2}
The present subsection is devoted to  recalling the results on category $\mathcal{O}$ and generalized Soergel functor $\mathbb{V}$ in \cite{Lo10b, Lo15}.
Choose a Levi subalgebra $(\ggg_{\bar{0}})_0 \subset \ggg_{\bar{0}}$, an $\mathfrak{sl}_2$-triple $(e,h,f) \subset (\ggg_{\bar{0}})_0$,  an integral element $\theta \in \mathfrak{z}((\ggg_{\bar{0}})_0)$ as in \cite[\S 2.6.1]{Lo15}. Recall that we have  used the grading $ \ggg_{\bar{0}}=\bigoplus_{i \in \mathbb{Z}}\ggg_{\bar{0}}(i)$ with respect to  $\ad(h)$ to define the W-algebra $\mathcal{W}_0$. We also need the grading $\ggg_{\bar{0}}=\bigoplus_{i \in \mathbb{Z}}(\ggg_{\bar{0}})_i$ with respect to  $\ad(\theta)$, where $(\ggg_{\bar{0}})_0 $ is exactly the Levi subalgebra introduced previously.
Let  $\mathfrak{p}$ be the parabolic subalgebra $\mathfrak{p}=(\ggg_{\bar{0}})_0(\geq 0) + (\ggg_{\bar{0}})_{>0}$.
Here $(\ggg_{\bar{0}})_0(\geq 0)$ (resp. $(\ggg_{\bar{0}})_{>0}$) stands for the subalgebra 
of $(\ggg_{\bar{0}})_0$ (resp. $\ggg_{\bar{0}}$)  generated by elements with non-negative
 (resp. positive) grading from $\ad(h)$ (resp. $\ad(\theta)$). Let $P$ be the corresponding  parabolic subgroup and $\mathcal{O}^{P}_\nu$ be the parabolic 
category $\mathcal{O}$ generated by finitely generated $(P,\nu)$-equivariant  $(\mathcal{U}_0,P)$-modules for a character $\nu$ of $\mathfrak{p}$.
Let $\mathfrak{t}=\mathfrak{z}((\ggg_{\bar{0}})_0)$ and $T \subset Q$ be the torus with $\mathrm{Lie}(T)=\mathfrak{t}$. Let $R$ stand for 
the centralizer of $T$ in $Q$.

View $\theta$ as an element of $\mathcal{W}_0$ by the embedding $\mathfrak{q} \hookrightarrow \mathcal{W}_0$.
 Let $ \mathcal{W}_0=\bigoplus_{\alpha \in \mathbb{Z}} (\mathcal{W}_0)_\alpha$ be the decomposition by eigenspaces of
 $\ad(\theta)$.  Set
 $$  (\mathcal{W}_0)_{\geq 0}=\bigoplus_{\alpha \geq 0}  (\mathcal{W}_0)_{\alpha}, (\mathcal{W}_0)_{ > 0}=\bigoplus_{\alpha > 0}  (\mathcal{W}_0)_{\alpha},
 (\mathcal{W}_0)_{ \geq 0}^{+}= (\mathcal{W}_0)_{ \geq 0} \cap  \mathcal{W}_0  (\mathcal{W}_0)_{> 0}.$$
 The following statements are main results of \cite{Lo10b}.
The quotient $(\mathcal{W}_0)^{0}:= (\mathcal{W}_0)_{ \geq 0}/(\mathcal{W}_0)_{ \geq 0}^{+}$ is isomorphic to the W-algebra arising from the pair $((\ggg_{\bar{0}})_0,e)$. For a finite dimensional simple $(\mathcal{W}_0)^{0}$-module $N$, define the \textit{Verma} module $\Delta_{\mathcal{W}_0}^{\theta}(N):=\mathcal{W}_0\otimes_{(\mathcal{W}_0)_{\geq 0}}N$. 
 Then $\Delta_{\mathcal{W}_0}^\theta(N)$ has a unique irreducible quotient $L_{\mathcal{W}_0}^\theta(N)$.
 Any finite dimensional irreducible $\mathcal{W}_0$-module can be obtained by this way. For a character $\nu$ of $R$,  $\mathcal{O}_{\theta} (\ggg_{\bar{0}},e)^{R}_\nu$ denotes the $(R,\nu)$-equivarinat category $\mathcal{O}$ defined for $\mathcal{W}_0$.

Let $\mathfrak{u}:= \mathfrak{p} \cap [f, \ggg_{\bar{0}}]$, which is a Lagrangian subspace of $V_{\bar{0}}$. Choose an $R \times \mathbb{K}^{*}$-equivariant embedding  
$\iota : V_{\bar{0}} \hookrightarrow \mathcal{U}_{0,\hbar}^{\wedge}$ as in \cite[\S 4.1.2]{Lo15}. We have an isomorphism  
\begin{equation} 
\Phi_{0,\hbar}: \mathbf{A}_{\hbar}^{\wedge}(V_{\bar{0}}) \otimes \mathcal{W}_{0,\hbar}^{\wedge} \longrightarrow \mathcal{U}_{0,\hbar}^{\wedge}	
\end{equation}
of quantum algebras from $\iota$  and  $(\mathcal{W}_{0,\hbar}^{\wedge})_{\mathbb{K}^*-l.f}/(\hbar-1)=\mathcal{W}_{0}$.

The generalized Soergel functor  $\mathbb{V}:\mathcal{O}^P_{\nu} \longrightarrow \mathcal{O}_{\theta} (\ggg_{\bar{0}},e)^{R}_\nu$ is defined by three different but equivalent ways in \cite{Lo15}. We recall the first one. For $M \in \mathcal{O}^P_\nu $,  let $M_{\hbar}^{\wedge_\chi}$ denotes the completion of Rees module $M_{\hbar}$ with respect to the
inverse image of the maximal ideal of $\chi$ under the homomorphism $(\mathcal{U}_{0})_{\hbar} \rightarrow S[\ggg_{\bar{0}}]$ given by $\hbar=0$. 
Let $M'_{\hbar} \subset  M_{\hbar}^{\wedge_\chi} $ be the annihilator of $\Phi_{0,\hbar}(\mathfrak{u})$. Then  $M'_{\hbar}$ is $\Phi_{0,\hbar}(\mathcal{W}_{0,\hbar}^{\wedge})$-stable, because $\Phi_{0,\hbar}(\mathcal{W}_{0,\hbar}^{\wedge})$ commutes with $\Phi_{0,\hbar}(\mathbf{A}_{\hbar}^{\wedge}(V_{\bar{0}})) \supset \Phi_{0,\hbar}(\mathfrak{u})$.   The generalized Soergel functor $\mathbb{V}$ is defined 
as follows 
$$\mathbb{V}(M):=(M'_{\hbar})_{\mathbb{K}^*-l.f}/(\hbar-1).$$
 There is  a rational action of $R$ on $\mathbb{V}(M)$ by the construction.  For the simple module $L(\lambda) \in \mathcal{O}^P_\nu$, we have  

\begin{equation}\label{equ4.1}
\mathbb{V}(L(\lambda))= \bigoplus_{i \in I_{\lambda}} L_{\mathcal{W}_0}^{\theta}({N^{0}_i}).	
\end{equation}
Here $L_{00}(\lambda)$ stands for the  finite dimensional $(\ggg_{\bar{0}})_{0}$-module with highest weight $\lambda$  and $N^{0}_i$ for $i \in I_{\lambda}$ run over the finite dimensional simple modules of $(\mathcal{W}_0)^{0}$ lying over $J_0(\lambda)=\Ann(L_{00}(\lambda))$.
 For simplicity, we write $L_{\mathcal{W}_0}^{\theta}({N^{0}_i})$ instead of $N_i$, $i \in I_{\lambda}$.

\subsection{Description of $\mathbb{V}(\widehat{L}(\lambda))$ for $\lambda \in \Lambda_{\mathfrak{p}}$}

Denote by $\mathcal{O}^P_\nu(\mathcal{U})$  the category of $\ggg$-supermodules lying in parabolic category  $\mathcal{O}^P_{\nu}$ for $\ggg_{\bar{0}}$.  Similarly, let $\mathcal{O}_{\theta} (\ggg_{\bar{0}},e)^{R}_{\nu}(\tilde{\mathcal{W}})$ be the category of $\tilde{\mathcal{W}}$-modules lying in $\mathcal{O}_{\theta}(\ggg_{\bar{0}},e)^{R}_{\nu}$.

The forthcoming Lemma \ref{Sky2} and Theorem \ref{Sorgel functor} follow from  \cite[\S 6.3.1]{Lo15}, which are given in the more general setting of Dixmier algebras.  For reader's convenience, we give a brief explanation in our special case.

Let  $\mathrm{Wh}(\ggg_{\bar{0}},e)^{R}_{\nu}$ be the category of $R$-equivariant  generalized Whittaker modules 
defined in \cite[\S 3.2.3]{Lo15}. This Whittaker category is similar to the one considered  in 
Theorem \ref{Skrabin equi thm}. It is defined by a nilpotent Lie subalgebra of $\ggg_{\bar{0}}$ different 
from $\mathfrak{m}_{\bar{0}}$. Let $\mathrm{Wh}(\ggg_{\bar{0}},e)^{R}_{\nu}(\mathcal{U})$ stand for the category of  $\ggg$-supermodules lying in $\mathrm{Wh}(\ggg_{\bar{0}},e)^{R}_{\nu}$. There is a generalized Skryabin's equivalence 
$ \mathcal{K}: \mathrm{Wh}(\ggg_{\bar{0}},e)^{R}_{\nu} \rightarrow \mathcal{O}_{\theta} (\ggg_{\bar{0}},e)^{R}_{\nu} $ with inverse $\mathcal{K}^{-1}$ ; see \cite[\S4]{Lo15} for the definition. It is easy to know that $\mathcal{K}$ sends $\mathcal{U}$-supermodules in $\mathrm{Wh}(\ggg_{\bar{0}},e)^{R}_{\nu}$
to  $\mathcal{\tilde{W}}$-supermodules in $\mathcal{O}_{\theta} (\ggg_{\bar{0}},e)^{R}_{\nu}$. The following lemma is an analog of Theorem \ref{Skrabin equi thm} and can be proved similarly.
\begin{lemma}\label{Sky2}
 There is a natural functor from   $\mathcal{O}_{\theta} (\ggg_{\bar{0}},e)^{R}_{\nu}(\tilde{\mathcal{W}})$ to  $\mathrm{Wh}(\ggg_{\bar{0}},e)^{R}_{\nu}(\mathcal{U})$
 induced by  $\mathcal{K}^{-1}$. 
\end{lemma}
The following result is crucial to  describe the image of simple objects in $\mathcal{O}^P_\nu(\mathcal{U})$  under $\mathbb{V}$. 
\begin{theorem}\label{Sorgel functor}
The functor $ \mathbb{V} :  \mathcal{O}^P_{\nu} \longrightarrow \mathcal{O}_{\theta} (\ggg_{\bar{0}},e)^{R}_{\nu}$ sends
simple $\mathcal{U}$-supermodules to simple objects in $\mathcal{O}_{\theta} (\ggg_{\bar{0}},e)^{R}(\tilde{\mathcal{W}})$.
 \end{theorem} 

\begin{proof}
By construction, we have that $\mathbb{V}$ restricts to 
a functor  from  $\mathcal{O}^P_{\nu}(\mathcal{U})$ to $ \mathcal{O}_{\theta} (\ggg_{\bar{0}},e)^{R}_{\nu}(\tilde{\mathcal{W}})$.	
Let $\mathbb{V}^{*}: (\ggg_{\bar{0}},e)^{R}_{\nu} \rightarrow \mathcal{O}^P_{\nu} $ be the right adjoin functor of $\mathbb{V}$ defined in \cite[Proposition 4.4]{Lo15}. The construction (precisely the last paragraph of pp.898 \cite{Lo15}) of  $\mathbb{V}^*$ conjunction with Lemma \ref{Sky2} implies that $\mathbb{V}^*$ sends $\tilde{\mathcal{W}}$-supermodules to $\mathcal{U}$-supermodules.  Furthermore,  $\mathbb{V}^*$ is restricted to a functor
 $ \mathcal{O}_{\theta} (\ggg_{\bar{0}},e)^{R}_{\nu}(\mathcal{\tilde{W}}) \rightarrow \mathcal{O}^P_{\nu}(\mathcal{U})$, which is the right adjoint to the restriction of $\mathbb{V}$. 
\end{proof}

\begin{theorem} \label{Thm sger func}
	For  $\lambda \in \Lambda_{\mathfrak{p}}$, recall that $N_i$, $i \in I_\lambda$  stand for the  simple $\mathcal{W}_0$-modules appearing in \eqref{equ4.1}.
Then we have 
$$\mathbb{V}(\widehat{L}(\lambda))=\bigoplus_{i \in I_\lambda} L^{K}_{\tilde{\mathcal{W}}}(N_i).$$
\end{theorem}
\begin{proof}
	
Since $L(\lambda) \subset \widehat{L}(\lambda)$, 
it follows that 
$$\bigoplus_{i} N_i \subset \mathbb{V}(\widehat{L}(\lambda)).$$
Note that the action of $\tilde{\mathcal{W}}_{+}^{\#}$ on $N_i$ for $i \in I_\lambda$ is trivial.
Now the theorem follows from  Proposition \ref{Triangular Decom} (3) and Theorem \ref{Sorgel functor}.
\end{proof}	
	
The following result implies that the action of $C_e$ on  $\gr.\mathrm{Irr}^{\mathrm{fin}}_{\lambda}(\tilde{\mathcal{W}})$ is transitive. 
\begin{corollary}\label{coro4.5}	
	For $\lambda \in \Lambda_{\mathfrak{p}}$,  the map  $ \mathrm{Irr}_{\lambda}^{\mathrm{fin}}(\mathcal{W}_0) \rightarrow \gr.\mathrm{Irr}_{\lambda}^{\mathrm{fin}}(\tilde{\mathcal{W}}):  N \mapsto L_{\mathcal{\tilde{W}}}^{K}(N)$ is bijective and  $C_e$-equivariant.
\end{corollary}

\begin{proof}
	
The main  result of \cite{Lo14}  states that
$$\mathrm{Irr}_{\lambda}^{\mathrm{fin}}(\mathcal{W}_0)=\{ ^g N \mid g \in C_e,  N=N_i \  \mbox{for some} \ i \in I_{\lambda} \}.$$	
Theorem \ref{Thm sger func} implies that  $L_{\mathcal{\tilde{W}}}^{K}(N) \in \gr.\mathrm{Irr}^{\mathrm{fin}}_{\lambda}(\tilde{\mathcal{W}})$. 
Hence we have 
$\mathrm{Irr}_{\lambda}^{\mathrm{fin}}(\tilde{\mathcal{W}})=\{ ^g L_{\mathcal{\tilde{W}}}^{K}(N) \mid g \in C_e\}$	
by \S \ref{sec2.5} and  Proposition \ref{Triangular Decom} (3). 

\end{proof}

 \subsection{ On $\ggg_{\bar{0}}$-rough structure of $\ggg$-supermodules} \label{sec4.5}

To compute the characters of $\mathcal{W}$-supermodules,
we need the expression  
\begin{equation}\label{equ 4.4}
\widehat{L}(\lambda)=\sum_{i \in S_{\lambda}} c_{i\lambda} \Delta_{P} (\lambda_i)
\end{equation} 
in the Grothendick group $K(\mathcal{O}^P_{\nu})$  of  the equivariant parabolic category $\mathcal{O}^P_{\nu}$ for $\ggg_{\bar{0}}$. 
Here $\Delta_{P}(\lambda_i)$ stands for the Verma module in $\mathcal{O}^P_{\nu}$ with the highest weight $\lambda_i$.
The coefficients  $c_{i\lambda}$ can be obtained from the $\ggg_{\bar{0}}$-rough structure of simple $\ggg$-modules by the following two ways.  

We may view $\hat{L}(\lambda)$ as a $\ggg_{\bar{0}}$-module and assume that 
$$\widehat{L}(\lambda)=\sum d_{\lambda\mu_i} L(\mu_i)$$
in $K(\mathcal{O}^P_{\nu})$. 
Here the coefficients $d_{\lambda\mu_i}$ are the multiplicities of $L(\mu_i)$ in $\widehat{L}(\lambda)$. 
However in general, it is still open  to determine  $d_{\lambda\mu_i}$. 
It can be computed by the Kazhdan-Lusztig theory of  Lie algebras when $\ggg=\mathfrak{gl}(m|n)$ and $\lambda$ is typical; see \cite{CM}.
For the recent progress on the rough structures for type \Rmnum{1} Lie superalgebras and their applications,  see also \cite{CCC,CCM,Ch}. Thus we can determine the coefficients  $c_{i\lambda}$ by Kazadan Lusztig theory of $\mathcal{O}^P_{\nu}$.

The coefficients  $c_{i\lambda}$  may  also be determined by the super parabolic Kazhdan-Lusztig theory which is still open in general presently.  Let 
 $\hat{\mathfrak{p}}$ be the parabolic sub-superalgebra $\mathfrak{p}+\ggg_{\bar{1}}(\geq 0)$, where 
 $\ggg_{\bar{1}}(\geq 0)$ defined by the similar way as above.
 Suppose that 
 $$\widehat{L}(\lambda)=\sum_{i \in \hat{S}_{\lambda}} \hat{c}_{i\lambda} \hat{\Delta}_{\hat{\mathfrak{p}}} (\lambda_i)$$
 in the Grothendick group of the super parabolic category $\mathcal{O}_{\hat{\mathfrak{p}}}$ for $\hat{\mathfrak{p}}$, where $\hat{\Delta}_{\hat{\mathfrak{p}}} (\lambda_i)$ is the parabolic Verma module in  $\mathcal{O}_{\hat{\mathfrak{p}}}$.   A filtration of Verma
 modules of Lie superalgebras by that of Lie algebras was given in \cite[Theorem 3.2]{Mu97b}. Generalizing this result to the parabolic case, we may find the coefficients $c_{i\lambda}$ in \eqref{equ 4.4}.
  
 \subsection{Algorithm for character formulas} \label{sec4.5}
Now we present our algorithm to compute the characters of modules in $\mathrm{Irr}_{\lambda}^{\mathrm{fin}}(\mathcal{W})$ for  $\lambda \in \Lambda_{\mathfrak{p}}$.
It was obtained in \cite[ Theorem 4.8 (\rmnum{4})]{Lo15} that 
\begin{equation}
 \Ch(\mathbb{V}(\Delta_{P}(\mu)))=\dim(L_{00}(\mu))e^{\mu-\rho} \prod_{i=1}^{k}(1-e^{\mu_i})^{-1}.
\end{equation} 
Here $\mu_i$, $i=1,2,\ldots, k$ are  the weights of $\mathfrak{t}$
in $(\ggg_{\bar{0}})_{<0} \cap \mathfrak{z}_{\ggg_{\bar{0}}}(e)$, $\rho$ is  the half of  sum of all the positive roots of $\ggg_{\bar{0}}$.
Applying  $\mathbb{V}$ to the both sides of \eqref{equ 4.4} and by \cite[Theorem 4.8]{Lo15}, we have 
\begin{equation}
 \Ch (\mathbb{V}(\widehat{L}(\lambda)))=\sum_{i \in S_{\lambda}} c_{i\lambda} \dim(L_{00}(\lambda_i))e^{\lambda_i-\rho} \prod_{i=1}^{k}(1-e^{\mu_i})^{-1}. 
\end{equation} 

Thanks to  Theorem \ref{Thm sger func}, we have that $\mathbb{V}(\widehat{L}(\lambda))$ is the direct sum of $|I_{\lambda}|$  simple $\mathcal{\tilde{W}}$-supermodules.
 These supermodules are transitive under the twist action of $Q_0/Q_0^\circ$ , where $Q_0$ is the centralizer of $\mathfrak{sl}_2$-triple $\{e,h,f\}$ in $(G_{\bar{0}})_0$.
 Note that we consider the character with respect to the torus $ \mathfrak{t}=\mathfrak{z}((\ggg_{\bar{0}})_0)$. Therefore they have the same characters. Thus 
 \begin{equation}
 	\Ch(L^{K}_{\tilde{\mathcal{W}}}(N_i))=|I_{\lambda}|^{-1}\sum_{i \in I_{\lambda}} c_{i\lambda} \dim(L_{00}(\lambda_i))e^{\lambda_i-\rho} \prod_{i=1}^{k}(1-e^{\mu_i})^{-1}.
\end{equation}

Now by \S \ref*{sec2.5} and Corollary \ref{coro4.5},  we obtain a character formula for all $M \in \gr.\mathrm{Irr}_\lambda^{\mathrm{fin}}(\mathcal{\tilde{W}})$. 
Note that by definition there is an   embedding $ \mathfrak{t} \hookrightarrow \mathcal{W} \hookrightarrow \mathcal{\tilde{W}}$. Proposition \ref{prop2.5} now yields
  $$ \Ch(L^{K}_{\tilde{\mathcal{W}}}(N_i)')=\Ch(L^{K}_{\tilde{\mathcal{W}}}(N_i))\prod_{i=1}^{l}(1+e^{\mu_i'})^{-1}.$$
Here   $(L^{K}_{\tilde{\mathcal{W}}}(N_i))'$ is the simple $\mathcal{W}$-supermodule obtained from $L^{K}_{\tilde{\mathcal{W}}}(N_i)$ (see  Proposition \ref{prop2.5}
) and $\mu_i',i=1,2,\ldots, l$  are the weights of the Lagrangian $\mathfrak{u}_{\bar{1}}^{*}$.

By Proposition \ref{Triangular Decom} and \cite{Lo15}, in order to compute  the characters of modules in $\gr.\mathrm{Irr}^{\mathrm{fin}}_{\lambda}(\mathcal{W})$,
we only need to determine the coefficients $c_{i\lambda}$ in
\eqref{equ 4.4}. This is a fundamental problem in the representation theory of  
Lie superalgebras.

\section*{Acknowledgements}
The author is partially supported by NFSC (grant No.11801113) and  RIMS, an international joint usage/research center located in Kyoto University. This work is motivated by communications with Tomuyuki Arakawa and a part of it is written during the author's visit to him at RIMS. The author is indebted much to him for many helpful discussions. He thanks Hao Chang for the tremendous help in improving the language.  He also thanks for the helpful communications from Bin Shu and comments from Yang Zeng. Finally, the author would like to express his gratitude to the editors and referees for their numerous valuable comments.


\begin{thebibliography}{ABC}
	\bibitem[BK]{BK} W. Borho, H. Kraft, {\em \"{U}ber die Gelfand-Kirillov-Dimension}, Math. Ann. 220(1976), 1-24.
	\bibitem[BBG]{BBG}  J. Brown, J. Brundan, S. Goodwin, {\em Principal W-algebras for $GL(m|n)$},  Algebra. Number Theory 7  (2013), 1849-1882.
	\bibitem[BG]{BG}  J. Brundan, S. M. Goodwin, {\em  Whittaker coinvariants for $GL(m|n)$}, Adv. Math.  347 (2019), 273-339.
   \bibitem[CCC]{CCC}   C. Chen, S. Cheng, K.Coulembier, {\em Tilting modules for classical Lie superalgebras},  J. Lond. Math. Soc. (2) 103 (2021), no. 3, 870-900. 
\bibitem[Ch]{Ch} C. Chen, {\em  Whittaker modules for classical Lie superalgebras},  Commun. Math. Phys. 388(2021), 351-383. 
\bibitem[CCM]{CCM} C. Chen, K. Coulembier, V. Mazorchuk,  {\em Translated simple modules for Lie algebras and simple supermodules for Lie superalgebras}, Math. Zeits (2020), 1-27.
\bibitem[ChM]{CM} C. Chen, V. Mazorchuk, {\em Simple supermodules over Lie superalgebras},  Trans. Amer. Math. Soc. 374 (2021), 899-921.
\bibitem[CoM]{CoM} K. Coulembier, I. Musson, {\em  The primitive spectrum for $\mathfrak{gl}(m|n)$}, Tohoku Math. J.(2)70 (2018), no. 2, 225-266.
\bibitem[DSK]{DSK} A. De Sole,  V. Kac, {\em  Finite vs  affine W-algebras}, Jpn. J. Math. 1(1)(2006), 137-261.
\bibitem[Di]{Di}J. Dixmier, {\em Enveloping algebras},  North-Holland Mathematical  Library, Vol. 14. 1977.
\bibitem[KRW]{KRW} V. Kac, S. S. Roan,  M. Wakimoto, {\em Quantum reduction for affine superalgebras}. Commun. Math. Phys. 241.2-3(2003), 307-342.
\bibitem[KL]{KL} G. R. Krause,  T. H. Lenagan, {\em Growth of algebras and Gelfand-Kirillov dimension}, revised edition, Graduate Studies in Mathematics, Vol. 22(2000), American Mathematical Society.
\bibitem[Le]{Le} E. Letzter,  { \em A bijection of primitive spectra for classical Lie superalgebras of Type \Rmnum{1}},  J. London Math. Soc. 53 (1996), 39-49.
\bibitem[Lo10a]{Lo10} I. Losev, {\em  Quantized symplectic actions and W-algebras}, J. Amer. Math. Soc. 23 (2010), 35-59.
\bibitem[Lo10b]{Lo10b} I. Losev, {\em  On the structure of the category $\mathcal{O}$ for W-algebras}, Seminaires. et. Congres 25(2010), 351-368.
   \bibitem[Lo11]{Lo11} I. Losev, {\em Finite dimensional representations of W-algebras}, Duke Math. J. 159 (2011), 99-143.
     \bibitem[Lo12]{Lo12} I. Losev, {\em  Primitive ideals in W-algebras of type A},  J. Algebra, 359(2012), 80-88.
   \bibitem[Lo14]{Lo14} I. Losev, V. Ostrik, {\em  Classification of finite dimensional irreducible modules over W-algebras}, Compos. Math, 150(2014), no.6, 1024-1076.
   \bibitem[Lo15]{Lo15} I. Losev, {\em  Dimensions of irreducible modules over W-algebras and Goldie ranks}, Invent. Math, 200(3)(2015), 849-923.
   \bibitem[MMOPS]{MMOPS} J. Marsden, G. Misiolek, J-P. Ortega, M. Perlmutter, T. Ratiu, {\em Hamiltonian reduction by stages}, Lecture Notes in Mathematics, 1913(2007).
    \bibitem[Mu92]{Mu92} I. Musson, { \em A classification of primitive ideals in the enveloping algebra of a classical simple Lie superalgebra},  Adv. Math 91.2(1992), 252-268.
    \bibitem[Mu97a]{Mu97} I. Musson, {\em   The enveloping algebra of the Lie superalgebra $\mathfrak{osp}(1|2r)$},  Rep. Theory 1(1997), 405-423.
   \bibitem[Mu97b]{Mu97b} I. Musson,  {\em  On the center of the enveloping algebra of a classical simple Lie superalgebra}, J. Algebra 193(1997)(1), 75-101.
  \bibitem[Mu12]{Mu12} I. Musson, {\em Lie superalgebras and enveloping algebras}, GSM131 (2012), Amer. Math.
    Society.
  \bibitem[Pe]{Pe} Y. Peng, {\em Finite W-superalgebras via super Yangians}, Adv. Math. 337(2021) (107459).
    \bibitem[Pr1]{Pr1} A. Premet,  {\em  Special transverse slices and their enveloping algebras}, Adv. Math. 170 (2002), 1-55.
   \bibitem[PS1]{PS1}   E. Poletaeva,  V. Serganova, {\em  On Kostant's theorem for the Lie superalgebra Q(n)}, Adv. Math. 300 (2016), 320-359.
\bibitem[SX]{SX} B. Shu, H. Xiao, {\em  Super formal Darboux-Weinstein theorems and finite W superalgebras}, J. Algebra 550(2020), 242-265.
 \bibitem[WZ]{WZ} W. Wang, L. Zhao, {\em Representations of Lie superalgebras in prime characteristic I}, Proc. London Math. Soc. 99 (2009), 145-167.
\bibitem[Zh]{Zh} L. Zhao, {\em  Finite W-superalgebras for queer Lie superalgebras}, J. Pure Appl. Algebra 218(2014), no. 7, 1184-1194.
\bibitem[ZS]{ZS} Y. Zeng, B. Shu,  {\em Minimal $W$-superalgebras and the modular representations of basic Lie superalgebras}, Publ. RIMS Kyoto Univ. 55 (2019), 123-188.
\end{thebibliography}
\end{document}